\documentclass[12pt,a4paper,reqno]{amsart}
\usepackage{latexsym}
\usepackage{amssymb}
\usepackage{enumitem}
\usepackage{hyperref}

\def \za{\alpha}

\def \zd{\delta}
\def \ze{\varepsilon}

\def \zh{\theta}

\def \zl{\lambda}
\def \zm{\mu}

\def \zx{\xi}

\def \zp{\pi}

\def \zr{\rho}

\def \zt{\tau}

\def \zf{\varphi}

\def \zq{\psi}
\def \zw{\omega}

\def \zF{\Phi}

\def \zlma{\ell}

\def \zsu{\sum}
\def \zpr{\prod}

\def \zin{\cap}

\def \zun{\cup}
\def \zung{\bigcup}

\def \zdi{\oplus}

\def \zte{\otimes}

\def \zpu{\cdot}
\def \zpor{\times}
\def \zci{\circ}

\def \zmei{\leq}
\def \zmai{\geq}
\def \zco{\subset}

\def \zpe{\in}

\def \zeq{\equiv}

\def \znoi{\neq}

\def \znope{\not\in}

\def \zpar{\partial}
\def \zinf{\infty}

\def \zfl{\rightarrow}

\def \zbv{\mid}

\def \z/{\over}

\newcommand {\CC}{\mathbb C}

\newcommand {\RR}{\mathbb R}
\newcommand {\TT}{\mathbb T}
\newcommand {\ZZ}{\mathbb Z}
\newcommand {\NN}{\mathbb N}
\newcommand {\A}{\mathcal A}
\newcommand {\B}{\mathcal B}
\newcommand {\D}{\mathcal D}
\newcommand {\F}{\mathcal F}
\newcommand {\G}{\mathcal G}
\newcommand {\GG}{\mathbf G}
\newcommand {\HH}{\mathbf H}
\newcommand {\e}{\mathbf e}

\newcommand {\g}{\mathbf g}
\newcommand {\h}{\mathbf h}
\newcommand {\Id}{\operatorname{Id}}

\newcommand {\co}{\colon}

\DeclareMathOperator*{\Interseccion}{\mathop{\bigcap}}

\def\mylabel#1{\label{#1}}   

\newtheorem{theorem}{Theorem}[section]
\newtheorem*{theorem*}{Theorem}
\newtheorem{lemma}[theorem]{Lemma}
\newtheorem{corollary}[theorem]{Corollary}
\newtheorem*{corollary*}{Corollary}
\newtheorem{proposition}[theorem]{Proposition}

\newtheorem*{example*}{Example}

\newtheorem{question}[theorem]{Question}

\theoremstyle{definition}
\newtheorem{remark}[theorem]{Remark}
\newtheorem{example}[theorem]{Example}

\newcommand{\Aut}{\operatorname{Aut}}
\newcommand{\Diff}{\operatorname{Diff}}

\newcommand{\GL}{\operatorname{GL}}

\title[Smooth actions with free point are determined by two vector fields]
{Smooth actions of connected compact Lie groups with a free point are
determined by two vector fields}

\author{F.J.~Turiel}
\address[F.J.~Turiel]{
Departamento de {\'A}lgebra, Geometr{\'\i}a y Topolog{\'\i}a,
Facultad de Ciencias,
Campus de Teatinos, s/n,
29071-M{\'a}laga, Spain}
\email[F.J.~Turiel]{turiel@uma.es}

\author{A.~Viruel}
\address[A.~Viruel]{
Departamento de {\'A}lgebra, Geometr{\'\i}a y Topolog{\'\i}a,
Facultad de Ciencias,
Campus de Teatinos, s/n,
29071-M{\'a}laga, Spain}
\email[A.~Viruel]{viruel@uma.es}

\begin{document}

\begin{abstract}
Consider a smooth action $\mathbf G\times M \rightarrow M$ of a compact connected Lie group
$\GG$ on a connected manifold $M$. Assume the existence of a point of $M$ whose isotropy
group has a single element (a free point). Then we prove that there exist two complete
vector field $X,X_1$ such that their group of automorphisms equals $\mathbf G$ regarded as  a
group of diffeomorphisms of $M$ (the existence of a free point implies that the action of
$\mathbf G$ is effective). Moreover, some examples of effective actions with no free point where this result fails are exhibited.
\end{abstract}

\maketitle

\section{Introduction} \mylabel{sec-1}

In our two preceding works \cite{TV1,TV2}, focused on the inverse Galois problem's framework, we proved that any effective action on a manifold by a finite group or a torus can be described by a single vector field. However, it is crucial to note that the generality of this result does not extend to compact connected Lie groups, as exemplified in Example 6.3 of \cite{TV2}.

As a result, a compelling and challenging question naturally arises: how do we determine the actions of these compact connected Lie groups?  To address this intriguing inquiry, we embark on exploring the possibility of employing a family of vector fields as a potential solution.

In this paper, we demonstrate that for any connected compact Lie group smoothly acting on a connected manifold, the group can always be equated to the group of automorphisms of a couple of vector fields defined on this manifold, given the existence of a free point (Theorem \ref{T11}). Furthermore, we provide illustrative examples of Lie group effective actions lacking a free point, and in which the action cannot be described by any family of vector fields (Theorem \ref{T73}).

To ensure clarity and consistency throughout the manuscript, we begin by establishing the notation and conventions that will be employed throughout the paper. Our fundamental reference sources for differential topology, differential geometry, and Lie group actions are as follows: \cite{HI} for differential topology, \cite{KN} for differential geometry, and \cite{PA} and \cite{WA} for Lie groups actions. It is assumed that the reader is already acquainted with our two preceding papers \cite{TV1,TV2}.

Throughout this work, we consider manifolds (without boundary) and their associated objects to be real and of class $C^\infty$, with actions on the left, unless otherwise stated.

Consider a diffeomorphism $F\colon M\rightarrow M'$ and a vector field $X$ on $M$. The notation $F_* X$ represents the vector field on $M'$ defined as follows:
$$(F_* X)(y)= F_* (X(F^{-1}(y)))$$
for every $y\zpe M'$. Here, $F_* \colon TM \rightarrow TM'$ refers to the differential of the mapping $F$ that acts on tangent vectors between the tangent bundles $TM$ and $TM'$.

Let $\mathcal{F}$ be a family of vector fields on an $m$-manifold $M$. The group of automorphisms of $\mathcal{F}$, denoted as $\text{Aut}(\mathcal{F})$, comprises the subgroup of diffeomorphisms of $M$ that preserve each element of $\mathcal{F}$.

Now, let us consider an action of a Lie group $\GG$ on the manifold $M$. A point $p\in M$ is termed {\em free} if its isotropy group (or stabilizer) reduces to the neutral element $\mathbf{e}$ of $\GG$. Observe that if $p$ is free, then all points within its orbit are also free. It is worth noting that the existence of a free point guarantees the action's effectiveness, and conversely, this holds true when $\GG$ is a torus (as stated in Proposition 7.1 of \cite{TV2}). Thus, for torus actions, there is no difference between being effective and the existence of free points.

Furthermore, it is important to recall that if the action is effective, then $\GG$ can be treated as a subgroup of the group $\text{Diff}(M)$, which consists of diffeomorphisms of $M$.

Then, our main result is:

\begin{theorem}\mylabel{T11}

Consider an action of a connected compact Lie group $\GG$ on a connected manifold
$M$. If the action of $\GG$ has a free point then there exist two complete vector fields
$X,X_1$ such that $\GG=\Aut(X,X_1 )$.
\end{theorem}

On the other hand, Theorem \ref{T73} demonstrates that the aforementioned hypothesis of the existence of a free point cannot be substituted by merely assuming the action to be effective (further details are provided below).

Consider a family of vector fields, denoted as $\F$, defined on the manifold $M$. In the context of this paper, we shall use the terms {\em $\F$ determines} or {\em $\F$ describes} to refer to the action of a transformation group $\GG$, if and only if $\GG$ is the automorphism group of $\F$, expressed as $\GG = \text{Aut}(\mathcal{F})$. It is important to note that under such circumstances, every element belonging to the family $\mathcal{F}$ is inherently $\GG$-invariant.

\subsection*{Organization of the paper:} This paper is structured into eight sections, with the first section serving as the introduction. Sections
\ref{sec-2} and \ref{sec-3} present the essential results concerning Lie groups and vector fields, which are needed later on. The main result is demonstrated in Sections \ref{sec-4} and \ref{sec-5}.

In Section \ref{sec-7}, it is shown that the assumption of the existence of a free point cannot be omitted. Specifically, it is showed that the natural action of $\operatorname{U}(n)$
on $\CC^n \zeq\RR^{2n}$, $n\zmai 2$, is described by two vector fields. However, no family of $\operatorname{U}(n)$-invariant vector fields can completely determine the natural action of $\operatorname{U}(n)$ on $\CC^n \smallsetminus\{0\}$ (see Example \ref{E74}).

Section \ref{sec-8} presents additional examples of compact connected linear groups.
For instance, we consider the natural action of $\operatorname{SO}(n)$ for $n\geq 3$ and $\operatorname{SU}(m)$ for $m\geq 3$. In these cases, it is shown that the action cannot be described by invariant vector fields. However, a different outcome is observed in the case of the symplectic group, where its natural action can always be determined using two invariant vector fields. We end this section with a result on the stability of some determinable actions
(Proposition \ref{P82} and Remark \ref{R83}).

The concluding section of this paper presents a curated list of open problems and questions that naturally arise from our work.

\section{Preliminaries on Lie groups} \mylabel{sec-2}

In this section, we will present some fundamental concepts related to Lie groups, which will be crucial for our subsequent discussions.

Throughout the
remainder of this paper $\GG$ will be a connected compact Lie group of dimension
$n$, $\e$ its neutral element, and $\G$ its Lie algebra of left-invariant vector fields. When required, we will treat $\G$ as $T_\e \GG$ using the vector isomorphism $X \in \G \mapsto X(\e) \in T_\e \GG$.
We say that a couple $X,Y$ of elements of $\G$ is {\em dense} if the connected Lie
subgroup $\HH$  corresponding to the Lie subalgebra $\mathcal{H}$ spanned by $X$ and $Y$ is dense within $\GG$.

Note that if the couple $X,Y$ is dense and $f\co\GG \zfl\RR$ is a function such that $Xf=Yf=0$,
then $f$ is constant.

We show that dense couples do always exist.

\begin{proposition}\mylabel{P21}
In the algebra $\G$ of a connected compact Lie group $\GG$ there exist dense couples.
\end{proposition}
\begin{proof}
According to a classical result by Schreier and Ulam \cite{SU}, if $\GG$ is both connected and compact, then there always exist elements $\g$ and $\h$ in $\GG$ such that the group generated by $\g$ and $\h$ is dense in $\GG$. In this case, since the exponential map of $\GG$ is surjective onto the group, there exist $X$ and $Y$ in the Lie algebra $\G$ such that $\exp(X)=\g$ and $\exp(Y)=\h$. It is evident that the couple $X, Y$ is dense.
\end{proof}

Given an abstract group $G$ acting on a set $S$ on the left (resp.\ right), and $g\zpe G$,
$L_g \co S\zfl S$ (resp.\ $R_g \co S\zfl S$) denotes the induced by left (resp.\ right) multiplication by $g$.
Observe that if $S=G$ endowed with the natural left and right $G$-actions, then the maps $L_g$ and $R_h$ commute for any $g,h\zpe G$. Therefore:

\begin{lemma}\mylabel{L22}
Let $\GG$ be a Lie group and $\g\zpe\GG$. If  $X$ is a left-invariant vector field on $\GG$, then
$(R_\g )_* X$ is the left-invariant vector field with initial condition
$$((R_\g )_* X)(\e)=(\mathbb A_{\g^{-1}})_* (X(\e))$$
where $\mathbb A_\h =L_\h \zci R_{\h^{-1}}$.

Similarly, if $X'$ is a right invariant vector field on $\GG$, then $(L_g )_* X'$ is right invariant and
$$((L_\g )_* X')(\e)=(\mathbb A_{\g})_* (X'(\e)).$$
\end{lemma}

Let $Z$ be a vector field defined on a product manifold $P_1 \zpor P_2$. We say that
$Z$ is {\em horizontal} if $(\zp_2 )_* Z=0$, and {\em strongly horizontal}
or {\em s-horizontal}  if $Z$ is horizontal  and preserves the foliation given by the second factor
(roughly speaking,  $Z$ is tangent to the first factor and independent of the
``second variable''). In a similar way one can define the notions of {\em vertical} and
{\em s-vertical} vector field on $P_1 \zpor P_2$.

Observe that given a vector field $Z_1$ on $P_1$ (resp.\ $Z_2$ on $P_2$) there exists
one and only one vector field on $P_1 \zpor P_2$, which we will still denote as $Z_1$ (resp.\ $Z_2$),
such that it is  s-horizontal (resp.\ s-vertical) and $(\zp_1 )_* Z_1 =Z_1$
(resp.\ $(\zp_2 )_* Z_2 =Z_2$).

In a more general setting, given a map $\zp \co P\zfl Q$ between manifolds, a vector
field $Y$ on $P$ is called {\em vertical} if $\zp_* Y=0$.

Let $A$ be an open subset of $\RR^k$ and $\GG$ be a Lie group. Consider a vector field $H$
on $A$ and a map $\zf\co A \zfl T_\e \GG$. Let
$F\co A\zpor\GG \zfl A\zpor\GG$ be the diffeomorphism given by $F(x,\g)=(x,\g\zpu \exp(\zf(x)))$. Then:

\begin{lemma}\mylabel{L23}
If $\zf$ takes values in an abelian subalgebra $\G'$ of $\G\zeq  T_\e \GG$, and $H$
is thought as an s-horizontal vector field on $A\zpor\GG$, then $F_* H=H+V$
where $V$ is vertical, and every $V(x,\_)$, $x\zpe A$, is the left-invariant vector
field with initial condition $V(x,\e)=(\zf_* H)(x)$.
\end{lemma}
\begin{proof}
Consider the obvious left and right $\GG$-actions on $A\zpor\GG$ given by $\h\zpu(x,\g)=(x,\h\g)$ and
$(x,\g)\zpu\h=(x,\g\h)$, and observe these actions commute too. Since $F(x,\g)=R_{\exp(\zf(x))}(x,\g)$,
then $F$ and $L_{\tilde\g}$, $\tilde\g\zpe\GG$, commute. Therefore:
$$(L_{\tilde\g} )_* (F_* H)=F_* ((L_{\tilde\g} )_*H)=F_* H$$
hence $(L_{\tilde\g} )_* V=V$, which implies that every $V(x,\_)$ is left-invariant.

Now we compute the initial conditions that determine $V(x,\_)$. Let $\GG'$ be the connected Lie subgroup of $\GG$ with
algebra $\G'$; set $\GG_0 =\overline{\GG'}$. Then $\GG_0$ is a connected abelian
closed Lie subgroup of $\GG$ and a regular submanifold. As $\exp(\zf(A))\zco\GG_0$,
it is enough to compute the initial condition for $F\co A\zpor\GG_0 \zfl A\zpor\GG_0$.
In other words, we may assume that $\GG$ is abelian and connected by replacing
it with $\GG_0$ if needed, thus we assume that $\GG=\RR^r \zpor\TT^s$.

On $\RR^r \zpor\TT^s$ consider coordinates
$(y,\zh)=(y_1 ,\dots,y_r ,\zh_1 ,\dots,\zh_s )$ and the vector fields
$\zpar/\zpar y_1 ,\dots,\zpar/\zpar y_r ,\zpar/\zpar \zh_1 ,\dots,\zpar/\zpar \zh_s$.
Then a vector field $U$ on $\RR^r \zpor\TT^s$ is invariant (left or right)
whenever there exist $a_1 ,\dots,a_r ,b_1,\dots,b_s \zpe \RR$ such that
$U=\zsu_{j=1}^r a_j \zpar/\zpar y_j +
\zsu_{\zlma=1}^s b_{\zlma}\zpar/\zpar \zh_\zlma$.

If $T_\e (\RR^r \zpor\TT^s )$ is thought as $\RR^r \zpor\RR^s$ through
$\zpar/\zpar y_1 ,\dots,\zpar/\zpar y_r ,\zpar/\zpar \zh_1 ,\dots,\zpar/\zpar \zh_s$, then
$\exp(v,w)=(v,p(w))$ where $p\co \RR^s \zfl\TT^s \zeq(\RR^s /2\zp\ZZ^s )$ is
the canonical projection. Thus $F(x,(y,\zh))={\big(}x,y+\zf_1 (x),\zh+p(\zf_2 (x)){\big)}$
where $\zf_1 =\zp_1 \zci\zf$ and $\zf_2 =\zp_2 \zci\zf$.

Finally, a straightforward computation shows that $V(x,\e)=(\zf_* H)(x)$.
\end{proof}

The following is a restatement of Exercise 9 found in \cite[p. 134]{WA}:

\begin{lemma}\mylabel{L24}
Let $\GG$ be a connected Lie group and $\zq\co\GG\zfl\GG$ be a diffeomorphism  such that:

\begin{enumerate}[label={\rm (\alph{*})}]
\item\mylabel{L24a}
$\zq$ maps left-invariant vector fields to left-invariant vectors fields (or alternatively right ones
to right ones).
\item\mylabel{L24b}
$\zq(\e)=\e$.
\end{enumerate}

Then $\zq$ is a Lie group isomorphism.
\end{lemma}

\section{Some useful results on vector fields} \mylabel{sec-3}

In this section, we introduce key results pertaining to vector fields that will be utilized subsequently. We follow the convention established in Section \ref{sec-2}, and employ the notation defined therein. Moreover, from now on $\zx = \zsu_{j=1}^k x_j \zpar/\zpar x_j$ denotes the radial vector field of $\RR^k$, where $k\geq 1$, endowed with the canonical coordinates $(x_1,\ldots,x_k)$ of $\RR^k$. As needed, we may also view $\zx$ as an $s$-horizontal vector field on $\RR^k \zpor \GG$.

Let $P$ be a regular submanifold of a manifold $Q$ and $Z$ be a vector field defined
on an open subset of $Q$ that includes $P$. We say that  $Z$
{\em is tangent to $P$ at order $1$} if:
\begin{enumerate}[label={\rm (\arabic{*})}]
\item\mylabel{D1}$Z$ vanishes at $P$.
\item\mylabel{D2} For every vector field $Y$ defined on an open subset $B$ of $Q$, the vector
field $[Z,Y]$ is tangent to $P$ at $B\zin P$.
\end{enumerate}

On $\GG$ consider a left-invariant vector field $V$ and think of it as an s-vertical
vector field on $\RR^k \zpor\GG$. Set $X=\zx +V$ on $\RR^k \zpor\GG$. Then
$\{0\}\zpor\GG$ is the set of those points whose $X$-trajectory has compact adherence.
Therefore, $\{0\}\zpor\GG$ is an invariant of $X$, and if $Y$ is a vector field commuting
with $X$, then $Y$ has to be tangent to $\{0\}\zpor\GG$.

On the other hand, consider a vertical vector field $V_1$
such that each $V_1 (x,\_)$, $x\zpe\RR^k$, is left-invariant, and consider a function
$h\co\RR^k \zfl\RR$ with compact support such that $j^{4}_0 h=0$
but $j^{5}_0 h\znoi 0$. Set $X_1 = V_1 +(h\zci\zp_1 )X$.

Let  $\mathcal L(X,X_1 )$ be the set of those vector fields $Y$ on $\RR^k \zpor\GG$
such that:

\begin{enumerate}[label={\rm (\alph{*})}]
\item\mylabel{Da}
$[X,Y]=0$
\item\mylabel{Db}
$[X_1 ,Y]$ is tangent to $\{0\}\zpor\GG$ at order $1$.
\end{enumerate}
Finally, let $\{\zh_1 ,\dots,\zh_n \}$ be a basis of the algebra of right
invariant vector fields of $\GG$, and think of each $\zh_r$, $r=1,\dots,n$, as an s-vertical vector field
on $\RR^k \zpor\GG$. Then:

\begin{lemma}\mylabel{L31}
If $V,V_1 (0,\_)$ is a dense couple of $\G$, then $\mathcal{L}(X,X_1 )$ is a Lie
subalgebra of the Lie algebra of vector fields on $\RR^k \zpor\GG$
of dimension $k^2 +n$ with basis
$$ \left\{x_j {\frac {\zpar} {\zpar x_\zlma}}, \zh_r  \right\},\quad
j,\zlma=1,\dots, k;\,\, r=1,\dots,n.$$
\end{lemma}

\begin{proof}
First note that, for each $Y\zpe\mathcal{L}(X,X_1 )$,
$[X_1 ,Y]$ is tangent to $\{0\}\zpor\GG$ at order $1$ if and only if $[V_1 ,Y]$ is so. Therefore, we shall prove the result assuming this last hypothesis.

Let $Y=\zsu_{j=1}^k f_j (x,\g)\zpar/\zpar x_j
+\zsu_{r=1}^n \zf_r (x,\g)\zh_r\in\mathcal L(X,X_1 )$.
Then $X$ and $Y$ commute, and $Y$ has to be tangent to $\{0\}\zpor\GG$,
thus $f_j (0,\g)=0$, for every $\g\zpe\GG$ and every $j=1,\dots,k$.

Since $V,V_1$ are left-invariant, then $X$, $V$, and $V_1$ commute with $\zh_1 ,\dots,\zh_n$.
Now, from $[X,Y]_{\zbv\{0\}\zpor\GG} =[V_1 ,Y]_{\zbv\{0\}\zpor\GG}=0$ it follows that
$V\zpu\zf_r$ and $V_1 \zpu\zf_r$, $r=1,\dots,n$, vanish on $\{0\}\zpor\GG$. But
$V,V_1$ is a dense couple so each $\zf_r$ is constant on $\{0\}\zpor\GG$.

A computation shows that
$$[X,Y]=\widetilde Y +\zsu_{r=1}^n (X\zpu\zf_r )\zh_r$$
where $\widetilde Y$ is a functional combination of $\zpar/\zpar x_1 ,\ldots,\zpar/\zpar x_k$.
Hence $X\zpu\zf_r =0$, i.e.\ $\zf_r$ is constant along the  $X$-trajectories. But the
$\za$-limits of $X$-trajectories are included in $\{0\}\zpor\GG$, set where $\zf_r$ is constant, so each
$\zf_r$, $r=1,\dots,n$, is constant.

Replacing $Y$ by $Y-\zsu_{r=1}^n\zf_r \zh_r$ allows us to assume
$Y=\zsu_{j=1}^k f_j (x,\g)\zpar/\zpar x_j $ where each $f_j$ vanishes on $\{0\}\zpor\GG$.

In turn from $[X,Y]=0$ it follows $X\zpu f_j =f_j$, $j=1,\dots,k$.

On the other hand
$$[V_1 ,Y]=\zsu_{j=1}^k (V_1 \zpu f_j )\zpar/\zpar x_j +W$$
where $W$ is vertical. Thus $V_1 \zpu f_j =0$, $j=1,\dots,k$,  on  $\{0\}\zpor\GG$
since $[V_1 ,Y]$ is tangent to this submanifold. Moreover as $[V_1 ,Y]$ is tangent to
$\{0\}\zpor\GG$ at order $1$, $U\zpu(V_1 \zpu f_j )=0$, $j=1,\dots,k$, on
 $\{0\}\zpor\GG$ for every vector field
 $U=\zsu_{\zlma=1}^k a_\zlma \zpar/\zpar x_\zlma$ with $a_1 ,\dots,a_k \zpe\RR$.

 In order to finish the proof, we must show that every $f_j$ is independent
 of $\g$ and linear on $x$. In other words, given $f\co\RR^k \zpor\GG\zfl\RR$ such that
 $f(\{0\}\zpor\GG)=(V_1 \zpu f)(\{0\}\zpor\GG)=0$, $X\zpu f=f$ and
 $(U\zpu(V_1 \zpu f_j ))(\{0\}\zpor\GG)=0$ for every constant vector field $U$,
 we need to prove that $f$ is independent of $\g$ and linear on $x$.

First, we consider the case $k=1$. Since $f$ vanishes on $\{0\}\zpor\GG$, then $f=x\zf$ for some function $\zf$. Now $X\zpu f=f$ becomes $X\zpu\zf=0$,
which implies that $\zf$ is constant along the $X$-trajectories.

Set $U=\zpar/\zpar x$. Then
$${\frac {\zpar(V_1 \zpu f)} {\zpar x}}=V_1 \zpu\zf
+x{\frac {\zpar(V_1 \zpu\zf)} {\zpar x}}$$
vanishes on $\{0\}\zpor\GG$, hence  $V_1 \zpu\zf$ does so.

As $X=V$ on $\{0\}\zpor\GG$, and the couple $V,V_1 (0,\_)$ is dense, then
$\zf$ must be constant on $\{0\}\zpor\GG$. But the $\za$-limits  of all the $X$-trajectories are
included in $\{0\}\zpor\GG$, so $\zf$ is constant and $f=ax$ for some $a\zpe\RR$. That is, $f$ is independent of $\g$ and linear on $x$, what concludes the case $k=1$.


We consider now the case $k\zmai 2$. Let $E$ be any vector line in $\RR^k$, thus $E\cong \RR$. Since:
\begin{itemize}
\item $X$ and $V_1$ are tangent to $E\zpor\GG$,
\item the restriction of $\zx$ to $E$ is still the radial vector
field
\item and one may choose $U$ to be a constant vector field tangent to $E$,
\end{itemize}
we apply the $1$-dimensional case above to get that $f\co E\zpor\GG \zfl\RR$ is independent of $\g$ and linear on $E$

Finally, as the union of all the vector lines $E$ equals $\RR^k$, it follows that
$f$ is independent of $\g$ and linear on $\RR^k$. Indeed, clearly $f$ is homogeneous
of degree 1 and therefore linear (see Remark \ref{Nuevo1} below). 
\end{proof}

\begin{remark}\mylabel{Nuevo1}
Let $A\subset\RR^k$ be an open ball centered at the origin and radius $r\in(0,\infty]$.
If $\zf\colon A \zfl\RR$ is homogeneous of degree $d$ then $\zpar\zf/\zpar x_1 ,
\dots,\zpar\zf/\zpar x_k$ are homogeneous of degree $d-1$. Therefore, when $d=1$
the partial derivative  $\zpar\zf/\zpar x_1 ,\dots,\zpar\zf/\zpar x_k$ are constant
and $\zf$ has to be linear.
\end{remark}

\begin{proposition}\mylabel{P32}
Let $F\co\RR^k \zpor\GG \zfl\RR^k \zpor\GG$ be an automorphism of $X=\zx+V$ that
preserves $X_1 =V_1 +(h\zci\zp_1 )X$ on an open neighborhood of  $\{0\}\zpor\GG$,
that is, there exists $A$, an open set neighborhood of $\{0\}\zpor\GG$,
 such that $F_* X_1 (x,\g)=X_1 \big(F(x,\g)\big)$
for any $(x,\g)\zpe A$. If $V,V_1 (0,\_)$ is a dense couple, then there exist
an isomorphism $\zf\co\RR^k \zfl\RR^k$ and an element $\boldsymbol{\zl}\zpe\GG$ such that
$$F(x,\g)=(\zf(x),\boldsymbol{\zl}\g)$$
for all $(x,\g)\zpe\RR^k \zpor\GG$.
\end{proposition}
\begin{proof}
The diffeomorphism $F$ induces a Lie algebra automorphism of $\mathcal L(X,X_1)$.
Let $\mathcal L_1$ denote the ideal with basis $\{x_j \zpar/\zpar x_\zlma\,|\, j,\zlma=1,\dots,k\}$, while $\mathcal L_2$ denotes the ideal with basis $\{\zh_1 ,\dots,\zh_n \}$.
Then $[\mathcal L_1 ,\mathcal L_2 ]=0$ and
$\mathcal L (X,X_1 )=\mathcal L_1 \zdi\mathcal L_2$.

Note that $Y\zpe\mathcal L (X,X_1 )$ has a zero if and only if $Y\zpe\mathcal L_1$.
Thus $\mathcal L_1$ is an invariant of $F$. Moreover, $\mathcal L_1$ gives rise to the
foliation $\g=\rm constant$ (first define it on $(\RR^k -\{0\})\zpor\GG$ and then extend it
by continuity to $\RR^k \zpor\GG$). Of course this foliation is an invariant of $F$.

On the other hand, $F$ has to map $\mathcal L_2$ to an ideal $\mathcal L'_2$
such that $\mathcal L (X,X_1 )=\mathcal L_1 \zdi\mathcal L'_2$ and
 $[\mathcal L_1 ,\mathcal L'_2 ]=0$. Therefore, there exist
 $Z_1 ,\dots,Z_n \zpe\mathcal L_1$ such that
 $\{\zh_1 +Z_1 ,\dots,\zh_n +Z_n \}$ is a basis of $\mathcal L'_2$ and
 $[Z_r ,\mathcal L_1]=0$, $r=1,\dots,n$. This implies that $Z_r =b_r \zx$, $b_r \zpe\RR$, since the center of $\mathcal L_1$ is spanned by $\zx$.

But the closure of the trajectories of any element in $\mathcal L_2$ is compact,
hence the closure of the trajectories of any element in $\mathcal L'_2$ has to be so.
In the case of $\zh_r +Z_r$, this last assertion is true only if $Z_r =0$.
In short $\mathcal L'_2 =\mathcal L_2$ and $\mathcal L_2$ is an invariant of $F$.
Consequently, the foliation $x=\rm constant$ associated to $\mathcal L_2$ is also
an invariant of $F$.

From the invariance of the foliations associated to $\mathcal L_1$ and $\mathcal L_2$
respectively, that is the foliations given by the factors of $\RR^k \zpor\GG$,
it follows the existence of two map $\zf\co\RR^k \zfl\RR^k$
and $\zq\co\GG \zfl\GG$ such that $F(x,\g)=(\zf(x),\zq(\g))$.

As $F$ preserves $X$,  $\zf$ must preserve $\zx$, which implies that $\zf$
is linear. Obviously $\zf$ is a bijection hence an isomorphism.

In turn $\zq$ induces a Lie algebra automorphism of the ideal $\mathcal L_2$, that is
of the Lie algebra of right invariant vector fields. Moreover, composing $\zq$ on the left with
$L_{\boldsymbol{\zl}^{-1}}$, where $\boldsymbol{\zl}=\zq(\e)$, one may suppose $\zq(\e)=\e$. In this case by Lemma
\ref{L24} $\zq$ is a Lie group isomorphism. Therefore, $\zq$ induces an isomorphism of $\G$.

As $F_* X=X$, $F_* X_1 =X_1$ around $\{0\}\zpor\GG$, and $X=V$ and
$X_1 =V_1 (0,\_)$ on $\{0\}\zpor\GG$, one concludes that $\zq_* V=V$ and
$\zq_* V_1 (0,\_)=V_1 (0,\_)$.

Let $\mathcal H$ be the Lie subalgebra of $\G$ spanned by $V,V_1 (0,\_)$, and
$\mathbf H$ the connected Lie subgroup of $\GG$ corresponding to $\mathcal H$.
Clearly $\zq_*Y=Y$ for all $Y\zpe\mathcal H$, so $\zq$ equals the identity on
$\mathbf H$. Finally, since $V;V_1 (0,\_)$ is a dense couple $\overline{\mathbf H} =\GG$,
hence $\zq=\Id$, and the result follows.
\end{proof}

\begin{remark}\mylabel{R33}
Notice that if $F$ is like in Proposition \ref{P32} and  $\zf$ is shown to be a multiple of the identity,
then $\zf=\Id$.

Indeed, as $F$ preserves the couple $X,X_1$ around $\{0\}\zpor\GG$,  then close to the
origin $\zf$ has to preserve the function $h\co\RR^k \zfl\RR$. Set $\zf=a\Id$.
Consider a vector line $E$ in $\RR^k$ such that the jet of order $5$ at the origin
of $h_{\zbv E}$ does not vanish. Endow $E$ with the coordinate $y$. Then
$h_{\zbv E}(y)=y^5 \tilde h(y)$ where $\tilde h(0)\znoi 0$.

As $h$ is $\zf$-invariant near the origin one has $h_{\zbv E}(ay)=h_{\zbv E}(y)$, hence
$(ay)^5 \tilde h(ay)=y^5 \tilde h(y)$; finally computing the fifth derivative at the origin
yields $5!\,a^5 \tilde h(0)=5!\,\tilde h(0)$, so $a=1$.

Note that the support of $h$ around the origin can be taken as small as desired. Even more,
$h$ can be replaced by $hh_1$ in the expression $X_1 =V_1 +hX$ provided
that $h_1 (0)\znoi 0$.
\end{remark}

Let $A\subset \RR^k$ be open, and $W$ be a vertical vector field defined on $A\zpor\GG$. It what follows we say that $W$ is {\em left (respectively right) invariant} if each
$W(x,\_)$, $x\zpe A$,  is left (right) invariant.

\begin{lemma}\mylabel{L34}
Let $\widetilde X$ be the vector field on $\RR^k \zpor\GG$ given by $\widetilde X=\widetilde\zx
+\widetilde V$ where $\widetilde\zx=\zsu_{j=1}^k f_j (x)\zpar/\zpar x_j $,
and $\widetilde V$ is a vertical left-invariant vector field. Assume that:

\begin{enumerate} [label={\rm (\arabic{*})}]
\item On $\RR^k$ it holds:
\begin{enumerate}[label={\rm (\alph{*})}]
\item $\widetilde\zx$ is complete.
\item $\widetilde\zx(0)=0$ and  its linear part at the origin is a positive multiple of identity.
\item The outset of the origin equals $\RR^k$, that is, the $\za$-limit of every trajectory
of $\widetilde\zx$ is the origin.
\end{enumerate}
\item
There is an abelian subalgebra $\G'\zco\G$ such that $\widetilde V (x,\_)\zpe\G'$
for every  $x\zpe\RR^k$.			
\end{enumerate}

Then there exist a diffeomorphism $F\co\RR^k \zpor\GG \zfl\RR^k \zpor\GG$, $b\zpe\RR^+$,
and $V\zpe\G$ such that:

\begin{enumerate} [label={\rm (\roman{*})}]
\item  $F$ commutes with the (natural) left action of $\GG$ on $\RR^k \zpor\GG$
and $V=\widetilde V (0,\_)$.
\item $F_* \widetilde X=b\zx+V$ when $V$ is thought as an s-vertical vector field.
\end{enumerate}
\end{lemma}

\begin{proof}
The first part of the proof of Lemma 2.4 in \cite{TV2} (see also \cite{SST} and 
\cite[Proposition 2.1]{TV1}) allows us to assume $\widetilde\zx=\zx$.
(Note that this fact is essentially a consequence of the Sternberg  linearization
theorem stated as Theorem \ref{Nuevo2} just after the end of this proof.)

Define $\zq\co\RR^k \zfl\G'\zco\G\zeq T_\e \GG$ by setting
$\zq(x)=\widetilde V (x,\e)$. By Lemma 2.1 in \cite{TV2} there is
$\zf\co\RR^k \zfl\G'$ such that $(\zf_* \zx)(x)=\zq(0)-\zq(x)$, $x\zpe\RR^k$.

Now consider the diffeomorphism $F\co\RR^k \zpor\GG \zfl\RR^k \zpor\GG$ given by
$F(x,\g)=(x,\g\zpu \exp(\zf(x)))$. By Lemma \ref{L23}, $F_* \zx=\zx+\zh$ where $\zh$ is
a vertical left-invariant vector field with initial condition
$$\zh(x,\e)=(\zf_* \zx)(x)=\zq(0)-\zq(x),\quad x\zpe\RR^k .$$

Moreover, since $\G'$ is abelian then $F_* \widetilde V=\widetilde V$ by Lemma \ref{L22}.
Therefore:
$$F_* (\zx+\widetilde V)(x,\e)=\zx(x)+\zq(0)-\zq(x)+\widetilde V(x,\e)=\zx(x)+\zq(0).$$
\end{proof}

\begin{remark}\mylabel{R35}
Observe that if $\widetilde X$ matches the hypotheses of Lemma \ref{L34}, then $\widetilde X'=(h\zci\zp_1 )\widetilde\zx + (h_1 \zci\zp_1 )\widetilde V$, where
$h\co\RR^k \zfl\RR$ is a positive and bounded function and any $h_1 \co\RR^k \zfl\RR$, fulfills the hypotheses of Lemma \ref{L34} too.
\end{remark}

We will now rephrase the classical Sternberg linearization theorem \cite[Theorem 1]{SST} within the context of vector fields.

\begin{theorem}[Sternberg linearization theorem]\mylabel{Nuevo2}
Let $0\zpe D\zco\RR^n$ be an open set, and let $X=X_0 +X_1$  be a vector field on $D$ where
$$X_0 =\zsu_{i,j=1}^n a_{ij}   x_i  \frac{\zpar}  {\zpar x_j}, \,\, a_{ij}\zpe\RR,\quad
\text{and}  \quad  X_1 =\zsu_{j=1}^n f_j  \frac{\zpar}  {\zpar x_j} $$
with $f_j (0)=0$, $j=1,\dots,n$, and $(\zpar f_j /\zpar x_i )(0)=0$, $i,j=1,\dots,n$.
Let $\zl_1,\dots,\zl_n$ be the eigenvalues of the square matrix $(a_{ij})$ defined by the coefficients of $X_0$. Then if
$$\zl_j \znoi \zsu_{r=1}^n k_r \zl_r$$
for every $j=1,\dots,n$ and every family $k_1 ,\dots, k_n$ of non-negative integers
such that ${k_1 +\dots +k_n >1},$ there exist two open sets $0\zpe D_1 \zco D$, $0\zpe D_2 \zco D$ and
a diffeomorphism  $F\co D_1 \zfl D_2$ such that $F(0)=0$ and $F_* X=X_0$.
\end{theorem}

Note that given any manifold S then the projection map $\zp_1 \co S\zpor\GG\zfl S$ is  a $\GG$-principal fibre
bundle for the natural left $\GG$-action on $S\zpor\GG$.

Consider a connection $\widetilde\D$ on a $\GG$-principal fibre bundle $\zp\co P\zfl Q$
and a point $q\zpe Q$. One will say that $\widetilde\D$ is a {\em product around}
$\zp^{-1}(q)$ if there are an open set $q\zpe A\zco Q$ and a fibre bundle
isomorphism (over $\Id\co A\zfl A$) between
$\zp\co\zp^{-1}(A)\zfl A$ and  $\zp_1 \co A\zpor\GG\zfl A$ such that $\widetilde\D$,
regarded on $\zp_1 \co A\zpor\GG\zfl A$, is given by
$\widetilde\D (u,\g)=T_u A\zpor\{0\}\zco T_{(u,\g)}$.

\begin{lemma}\mylabel{L36}
Let $\zp\co E\zfl\RR^k$ be a $\GG$-principal fibre bundle endowed with a connection
$\D$. Let $\zx'$ be  the lift of $\zx$ to $E$ by means of $\D$. If around $\zp^{-1}(0)$
the connection $\D$ is a product then there exists a fibre bundle isomorphism
(over the identity) $F\co E\zfl\RR^k \zpor\GG$
such that $F_* \zx'=\zsu_{j=1}^k x_j \zpar/\zpar x_j$.
\end{lemma}

\begin{proof}
Since around $\zp^{-1}(0)$
the connection $\D$ is a product then there is an open ball $B_{2a}(0)$ such that, as principal fibre bundles,
$\zp\co\zp^{-1}(B_{2a}(0))\zfl B_{2a}(0)$ is identified (over the identity)  to
$\zp_1 \co B_{2a}(0)\zpor\GG\zfl B_{2a}(0)$ while $\D$ is given by the first factor. Note that on
$B_{2a}(0)\zpor\GG$ one has $\zx'=\zsu_{j=1}^k x_j \zpar/\zpar x_j$.

Define $\zt\co B_{2a}(0)\zfl \zp^{-1}(B_{2a}(0))\zeq  B_{2a}(0)\zpor\GG$ by setting
$\zt(x)=(x,\e)$. Let $\zF'\co\RR\zpor E\zfl E$ be the flow of $\zx'$ (clearly this vector field
is complete since $\GG$ is compact). Recall that the flow of $\zx$ is $\zF(t,x)=e^t x$.

Now define $\zt'\co\RR^k \smallsetminus \overline B_a (0)\zfl E$ by setting
$$\zt'(x)=\zF'\left(\log(\| x\|-a), \zt\left({\frac {ax} {\| x\|}}\right)\right).$$

Since $\zF(\log(\| x\| -a), ax/\| x\|)=x$ and $\zx$ is the projection of $\zx'$, it follows
that $\zt'(x)\zpe\zp^{-1}(x)$; that is $\zt'$ is a section.

As $\zt$ and $\zt'$ agree on $B_{2a}(0)\smallsetminus\overline B_a (0)$ they define a global section
$s\co\RR^k \zfl E$.

By construction $s(\RR^k )$, which is a closed regular submanifold of $E$, is union of integral
curves of $\zx'$, so $\zx'$ is tangent to $s(\RR^k )$.

Consider $\zq\co\RR^k \zpor\GG\zfl E$ given by $\zq(x,\g)=\g\zpu s(x)$. Then
$\zp_1 =\zp\zci\zq$. Thus $(\zq^{-1})_* \zx'$ projects by $\zp_1$ in $\zx$.
Moreover, $(\zq^{-1})_* \zx'$ is $\GG$-invariant and tangent to $\RR^k \zpor\{\e\}$
since $\zx'$ is $\GG$-invariant and tangent to $s(\RR^k )$. Therefore, necessarily
$(\zq^{-1})_* \zx'=\zsu_{j=1}^k x_j \zpar/\zpar x_j$, and it suffices to set $F=\zq^{-1}$.
\end{proof}

\begin{lemma}\mylabel{L37}
Let $Y$ be a vector field on $\RR^k \zpor\GG$ such that $Y=\zx+W$ where $W$ is vertical and left
invariant. Given $0<a<b$ there exists a principal bundle automorphism (over the identity)
$F\co\RR^k \zpor\GG \zfl\RR^k \zpor\GG$ such that $F_* Y=\zx+\widetilde W$ where:
\begin{enumerate}[label={\rm (\alph{*})}]
\item $\widetilde W$ is vertical and left-invariant.
\item $\widetilde W(x,\_)=0$ if $x\znope B_b (0)$.
\item $F$ equals the identity on $\overline B_a(0) \zpor\GG$, and therefore
$\widetilde W(x,\_)=W(x,\_
)$ if $x\zpe\overline B_a(0)$.
\end{enumerate}
\end{lemma}

\begin{proof}
Take $c\zpe (a,b)$.  Consider a vertical left-invariant vector field $W'$, with support
included in $B_b (0)\zpor\GG$, such that $W'=W$ on $\overline B_c(0) \zpor\GG$.
Let $\zF'$ be the flow of $\zx+W-W'$ and let
$s\co\RR^k \smallsetminus\overline B_a (0) \zfl\RR^k \zpor\GG$ be the section given by
$$s(x)=\zF'\left(\log(\| x\|-a), \left({\frac {ax} {\| x\|}},\e\right)\right).$$

As $s(x)=(x,\e)$ when $x\zpe B_c (0)\smallsetminus\overline B_a (0)$, we can extend $s$ to
$\RR^k$ by setting $s(x)=(x,\e)$ on $\overline B_a (0)$.

Now consider the diffeomorphism $\zq\co\RR^k \zpor\GG \zfl\RR^k \zpor\GG$ given
by $\zq(x,\g)=\g\zpu s(x)$. The same reasoning as in the proof of Lemma \ref{L36}
shows that
$$(\zq^{-1})_* (\zx+W-W')=\zsu_{j=1}^k x_j \zpar/\zpar x_j .$$
Therefore,
$$(\zq^{-1})_* Y=(\zq^{-1})_* (\zx+W-W')+(\zq^{-1})_* W'=\zx+ (\zq^{-1})_* W'.$$

For finishing set $F=\zq^{-1}$ and $\widetilde W= (\zq^{-1})_* W'$.
\end{proof}

\begin{remark}\mylabel{R38}
If $Y=\zx+W$ and $W$ vanishes around $\{0\}\zpor\GG$, then there is an $F$ as in Lemma
\ref{L37} such that $F_* Y=\zx$.

Indeed, choose $b>0$ in such a way that $W$ vanishes on $B_b (0)\zpor\GG$ and set
$W'=0$ in the proof above.
\end{remark}

\section{The free case} \mylabel{sec-4}
In this section we prove Theorem \ref{T11} under the following two additional hypotheses:
\begin{itemize}
\item The action of the $n$-dimensional connected compact Lie group $\GG$ on the
$m$-manifold $M$ is free.
\item $k=m-n\zmai 1$.
\end{itemize}

Observe that the free $\GG$-action on $M$ gives rise to a $\GG$-principal fibre bundle $\zp\co M\zfl B$ where $B$
is a connected $k$-manifold.

Let  $\zm\co B\zfl\RR$ be a proper and non-negative Morse function. Let $C$ denote the set of
critical points of $\zm$, which is closed and discrete, hence countable. From the paracompactness
of $B$ it follows that the exists a locally finite family $\{A_p \}_{p\zpe C}$ of disjoint
open sets such that $p\zpe A_p$, $p\zpe C$.

Accoring to \cite[Section 3]{TV2}, there exists a Riemannian
metric $\tilde g$ on $B$ such that the gradient vector field $Y$ of $\zm$ is complete and,
moreover, around each $p\zpe C$ there are coordinates $(x_1 ,\dots,x_n )$ with
$p\zeq 0$ and $Y=\zsu_{j=1}^k \zl_j x_j \zpar/\zpar x_j$ where:
\begin{enumerate}[label={\rm (\arabic{*})}]
\item\label{item:1} $\zl_1 =\dots=\zl_k >0$ if $p$ is a source, i.e.\ $p$ is a minimum of $\zm$.
\item\label{item:2}  $\zl_1 =\dots=\zl_k <0$ if $p$ is a sink, i.e.\ $p$ is a maximum of $\zm$.
\item\label{item:3} Some $\zl_j$ are positive and the others negative if $p$ is a saddle point.
\end{enumerate}
Indeed, the scalars $\zl_1, \ldots, \zl_k$ do depend on the point $p.$ However, to maintain a more concise notation, we have omitted this detail.

Your prompt looks good; however, I can make a small adjustment for clarity:

Furthermore, it is important to emphasize that these properties remain valid even when the vector field $Y$ is scaled by a positive bounded function. This is because their dependence is solely on the Sternberg linearization theorem (see Theorem \ref{Nuevo2}).

Let $I$ be the set of sources of $Y$, i.e. $I$ is the set of local minima of $\zm$. For each $i\zpe I$, let $S_i$ denote the
$Y$-outset of $i$. Then \cite[Lemma 3.3]{TV1} becomes:

\begin{lemma}\mylabel{L41}
The family $\{S_i \}_{i\zpe I}$ is locally finite and $\zung_{i\zpe I} S_i$ a dense set in $B$.
\end{lemma}

In what follows we assume that the locally finite family $\{A_p \}_{p\zpe C}$ defined above verifies
 $A_i \zco S_i$ by replacing $A_i$ by $A_i \zin S_i$
if necessary.

Shrinking each $A_p$, $p\zpe C\smallsetminus I$, allows us to identify $A_p$ with the ball
$B_{\ze_p}(0)$, with $p\zeq 0$ and $Y=\zsu_{j=1}^k \zl_j x_j \zpar/\zpar x_j$, thus 
the $\GG$-principal fibre bundle $\zp\co\zp^{-1}(A_p ) \zfl A_p$ becomes
 $\zp_1 \co B_{\ze_p}(0)\zpor\GG \zfl B_{\ze_p}(0)$.

As the family $\{A_p \}_{p\zpe C}$ is locally finite and its elements are disjoint, then it is possible to
find a connection $\D$ on $\zp\co M\zfl B$ with the following properties:

\begin{enumerate}[label={\rm (\arabic{*})}]
\item\mylabel{er1} For every $p\zpe C\smallsetminus I$ there exists $0<\zd_p \zmei\ze_p$ such that the connection $\D$ on
$B_{\zd_p}(0)\zpor\GG$ equals the one induced by the first factor.
\item $\D$ is a product on each fibre $\zp^{-1}(i)$, $i\zpe I$.
\end{enumerate}

Let $Y'$ denote the lift of $Y$ to $M$ by means of $\D$. By construction $Y'$ is
$\GG$-invariant and complete (recall $\GG$ is compact). on the other hand, given that Property
\ref{er1} above holds for $\D$, whenever $p\zpe C\smallsetminus I$ , we can assert that $Y'=\zsu_{j=1}^k \zl_j x_j \zpar/\zpar x_j$ on
$B_{\zd_p}(0)\zpor\GG$.

As it is well known, the $\zw$-limit of a regular trajectory of the gradient vector field $Y$
is empty, a saddle or a sink. Therefore (see \cite[p.\ 884]{TV2} for the definition of {\em order}):

\begin{lemma}\mylabel{L42}
The $\zw$-limit of a regular trajectory of $Y'$ is empty or a zero of $Y'$ of order $1$.
\end{lemma}

Now, according to \cite[Proposition 2.1]{TV1}, each outset $S_i$ may be identified to $\RR^k$ in such a way that
$Y=a_i \zsu_{j=1}^k  x_j \zpar/\zpar x_j$, $a_i >0$.
In turn, by Lemma \ref{L36}, $\zp\co\zp^{-1}(S_i )\zfl S_i \zeq\RR^k$ can be identified, as
$\GG$-principal fibre bundle, to $\zp_1 \co\RR^k \zpor\GG \zfl\RR^k$ in such a way that
$Y'=a_i \zsu_{j=1}^k  x_j \zpar/\zpar x_j$ (where
$a_i \zsu_{j=1}^k  x_j \zpar/\zpar x_j$ is regarded as an s-horizontal vector field).

Moreover, we can assume that $A_i$ is identified with an open set in $\mathbb{R}^k$ that includes the closed ball $\overline{B}_3(0)$. This can be achieved by applying an appropriate dilation to both instances of $\RR^k,$ the total space and the base space, as needed.

Let $K_i$ denote the compact set $\overline B_3 (0)\zpor\GG$  when regarded within $\zp^{-1}(S_i)$.

\subsection{Construction of $Y''$} \mylabel{sub-1}
The next step is to modify $Y'$ by adding a vertical component to it. More precisely, considering that the family ${\zp^{-1}(S_i)}{i\in I}$ is locally finite and its elements are disjoint (due to the properties of the family ${S_i}{i\in I}$), we will construct the new vector field by augmenting $Y'$ on each $\zp^{-1}(S_i)$ with a left-invariant vertical vector field whose support is contained within $K_i$.

Let $Z=\zx+U$ be a vector field on $\RR^k \zpor\GG$ where:

\begin{enumerate}[label={\rm (\arabic{*})}]
\item $U$ is s-vertical and left-invariant.
\item $U(0,\e)$ belongs to some dense couple of $\G\zeq T_\e \GG$.
\end{enumerate}

Let $F\co\RR^k \zpor\GG \zfl\RR^k \zpor\GG$ be the principal bundle automorphisms (thus a diffeomorphism) as provided by Lemma
\ref{L37} for $a=1$, $b=2$ and $Z$. Set $F_* Z=\zx+\widetilde U$.
Observe that the support of $\widetilde U$ is included in $\overline B_3 (0)\zpor\GG$. Let
$\widetilde U_i$ be the vector field on $\zp^{-1}(S_i )$ which is identified to
$\widetilde U$ on $\RR^k \zpor\GG$.

Define the vector field $Y''$ by setting:
\begin{itemize}
\item $Y''=Y'+a_i \widetilde U_i$ on $\zp^{-1}(S_i )$, $i\zpe I$.
\item $Y''=Y'$ on $M-\zung_{i\zpe I}\zp^{-1}(S_i )$.
\end{itemize}

Then $Y''$ is left-invariant and complete ($\zp_* Y''=Y$ and $\GG$ is compact). Moreover, Lemma
\ref{L42} still holds for $Y''$ because $Y''=Y'$ on $M\smallsetminus\zung_{i\zpe I}K_i$.

The identification of $\zp^{-1}(S_i )$ with $\RR^k \zpor\GG$ may be composed with
$F^{-1}$ in order to obtain a second identification of $\zp^{-1}(S_i )$ with $\RR^k \zpor\GG$
such that $Y''=a_i (\zsu_{j=1}^k x_j \zpar/\zpar x_j +U)$.

It is easily seen that $q\zpe\zp^{-1}(I)$ if and only if the closure of the $Y''$-trajectory
of $q$ is a torus of dimension $\zmai 1$. Thus, $\zp^{-1}(I)$ is an invariant of $Y''$, but it's important to note that not necessarily each $\zp^{-1}(i),$ $i\in I$, is an invariant of this vector field. To avoid this issue, we will need to modify $Y''.$

\subsection{Construction of $X$} \mylabel{sub-2}
{\em We first assume} $k\zmai 2$. For every source $i\zpe I$ choose $P_i \zco A_i$ a set of $k+1$ points, all them close enough to $i$ but different from it, in such a way that the linear
$\za$-limits of the $Y$-trajectories of these points are in general position (see \cite[pp.\ 319--320]{TV1}
and \cite[p.\ 844]{TV2} for definitions). Let $P=\zung_{i\zpe I}P_i$.

Fix an injective set theoretical map $\eta\colon  P\to\NN\smallsetminus\{0\}$. As $\{A_i \}_{i\zpe I}$ is still
locally finite, there is a bounded function $\zt\co B\zfl\RR$ that is positive on $B\smallsetminus P$ and has a zero of order $2\eta(p)$ at every $p\zpe P$.

Finally define $X=(\zt\zci\zp)Y'',$ which is  clearly a $\GG$-invariant and complete vector field.

Let $R_i$, $i\zpe I$, be the outset of $i$ with respect to $\zt Y$. Given a point $a\zpe\RR^k,$
let $[a,\zinf)\subset \RR^k$ denote the ray $\{ta \co t\zpe[1,\zinf)\}$. Identifying $S_i$ with $\RR^k$ in the usual way
implies that $S_i\smallsetminus R_i =\zung_{p\zpe P_i}[p,\zinf)$. Therefore:

\begin{lemma}\mylabel{L43}
The family $\{R_i\}_{i\in I}$ is locally finite, and the set $\bigcup_{i\in I} R_i$ is dense in $B$. Consequently, the family $\{\zp^{-1}(R_i)\}_{i\in I}$ is also locally finite, and the set $\bigcup_{i\in I} \zp^{-1}(R_i)$ is dense in $M$.
\end{lemma}

Now, we identify $\zp^{-1}(S_i )$ with $\RR^k \zpor\GG$ so that
$Y''=a_i (\zsu_{j=1}^k x_j \zpar/\zpar x_j +U)$. Then
$\zp^{-1}(S_i )\smallsetminus\zp^{-1}(R_i )=(\zung_{p\zpe P_i}[p,\zinf))\zpor\GG$.
According to \cite[Proposition 2.1]{TV1}, which can also be considered as a consequence of Theorem \ref{Nuevo2},
for each $i\zpe I$ there exists a diffeomorphism $f_i \co R_i \zfl\RR^k$ such that
$$(f_i )_* (a_i \zt\zsu_{j=1}^k x_j \zpar/\zpar x_j )=b_i \zsu_{j=1}^k x_j \zpar/\zpar x_j$$
where $b_i =a_i \zt(0)>0$. Therefore, the diffeomorphism
$F_i \co R_i \zpor\GG\zfl  \RR^k \zpor\GG,$ given by $F_i (x,\g)=(f_i (x),\g),$ allows us to
identify $\zp\co \zp^{-1}(R_i )\zfl R_i$
with $\zp_1 \co\RR^k \zpor\GG \zfl \RR^k$, as $\GG$-principal fibre bundles and over $f_i$, in such a way that
$$X=b_i \zsu_{j=1}^k x_j \zpar/\zpar x_j +\widetilde W_i$$
where $\widetilde W_i$ is vertical and left-invariant, $\widetilde W_i (0,\e)$ belongs to
some dense couple and every  $\widetilde W_i (x,\e)$, $x\zpe\RR^k$, is proportional
to  $\widetilde W_i (0,\e)$.

By Lemma \ref{L34} applied to $b_i \zsu_{j=1}^k x_j \zpar/\zpar x_j +\widetilde W_i$ there
exist a diffeomorphism $\widetilde F \co\RR^k \zpor\GG \zfl\RR^k \zpor\GG$, which
commutes with the left action of $\GG$, and $V_i \zpe\G$ such that
$V_i =\widetilde W_i (0,\_)$ and
$$\widetilde F_* \left(b_i \zsu_{j=1}^k x_j \zpar/\zpar x_j +\widetilde W_i \right)
=b_i \zsu_{j=1}^k x_j \zpar/\zpar x_j +V_i .$$

Composing, on the left, the foregoing identification with $\widetilde F$ gives rise to a second
identification
of $\zp\co \zp^{-1}(R_i )\zfl R_i$  with $\zp_1 \co\RR^k \zpor\GG \zfl \RR^k$, over some $\tilde f_i \co R_i \zfl \RR^k$ and as $\GG$-principal fibre bundles, such that
$X=b_i \zsu_{j=1}^k x_j \zpar/\zpar x_j +V_i$, $b_i >0$, where:
\begin{enumerate} [label={\rm (\arabic{*})}]

\item $V_i$ is s-vertical and left-invariant.
\item $V_i (0,\e)$ belongs to some dense couple.
\end{enumerate}
In the following, we will refer to this type of identification as {\em suitable}.

We also obtain that:
\begin{enumerate} [label={\rm (\alph{*})}]
\item Each $\zp^{-1}(i)$ is a connected component of the set of those points of $M$
such that the closure of their $X$-trajectories is a torus of dimension $\zmai 1$.
\item $\zp^{-1}(R_i )$ is the outset of $\zp^{-1}(i)$, that is the set of points of $M$ whose
$X$-trajectory has its $\za$-limit included in $\zp^{-1}(i)$.
\item For each $p\zpe P_i$ there exists one and only one open vector half-line $E_p$ such that
$E_p \zpor\GG$ is the set of points of $M$ whose $X$-trajectory has its $\za$-limit
included in $\zp^{-1}(i)$ and whose $\zw$-limit is a zero of $X$ of order $2\eta(p)$.
\end{enumerate}

(Recall that
$\zp^{-1}(R_i )=\big(\RR^k \smallsetminus\zung_{p\zpe P_i}[p,\zinf)\big)\zpor\GG$ and
$X=(\zt\zci\zp)\zpu(a_i \zsu_ {j=1}^k x_j \zpar/\zpar x_j +U)$ for some identification
of $\zp^{-1}(S_i )$ with $\RR^k \zpor\GG$.)

\begin{remark}\mylabel{R44}
From the preceding properties, it follows that if $F\co M\zfl M$ is an automorphism of $X$, then
$F\big(\zp^{-1}(i)\big)=\zp^{-1}(i)$ and, consequently, $F\big(\zp^{-1}(R_i )\big)=\zp^{-1}(R_i )$.

Moreover, if $F\colon\zp^{-1}(R_i) \rightarrow \zp^{-1}(R_i)$, when $\zp^{-1}(R_i)$ is suitably identified as $\mathbb{R}^k \times \GG$, is given by $F(x,\g)=(\zf_i (x),\boldsymbol{\zl}_i \zpu\g)$ where $\zf_i \zpe \GL(k,\RR)$ and $\boldsymbol{\zl}_i \zpe\GG$ (refer to Proposition \ref{P32}), then $\zf_i$ is a positive multiple of the identity. This is required to preserve the half-lines $E_p$, where $p \in P_i$, and these half-lines are in general position as a result of the construction of $P_i$.
\end{remark}

{\em Finally assume $k=1$}. In this case one cannot add additional zeros to construct the set $P$
because in this scenario $\zung_{i\zpe I}R_i$ is no longer dense.
Instead of that, we consider a proper and non-negative Morse function $\zm\co B\zfl\RR,$ which has two or more maxima if $B$ is not compact, 
and define $P$ to be the set of sinks of $Y$. Then, given an injective set theoretical map $\eta\colon  P\to\NN\smallsetminus\{0\}$, one can construct a bounded function $\zt\co B\zfl\RR$ that is positive on $B\smallsetminus P$ and has a zero of order $2\eta(p)$ at every $p\zpe P$, as in the case $k\geq 2$. Therefore, if $p\zpe P$, then $\zt Y$ has
 a zero of order $2\eta(p) +1$ at $p$.

The remaining arguments work in a similar way and one concludes that
$F\big(\zp^{-1}(i)\big)=\zp^{-1}(i)$ and $F\big(\zp^{-1}(R_i )\big)=\zp^{-1}(R_i )$ where $R_i =S_i$.
Obviously any $\zf_i \zpe \GL(1,\RR)$ is a multiple of the identity.

Another way of dealing with the case $k=1$ is to consider a Morse function $\zm$ on $B$
with a single minimum. In this case $F\big(\zp^{-1}(R_i )\big)=\zp^{-1}(R_i )$
where $R_i =S_i$.

\subsection{Construction of $X_1$} \mylabel{sub-3}
Let $V_1$ be a vertical vector field on $M$ such that 
\begin{enumerate} [label={\rm (\roman{*})}]
\item $V_1$
vanishes on $M\smallsetminus\zung_{i\zpe I}\zp^{-1}(R_i ).$

\item On every $\zp^{-1}(R_i )$, suitably identified to $\RR^k \zpor\GG$, satisfies:
\begin{enumerate} [label={\rm (\arabic{*})}]
\item $V_1$ is left-invariant and its support is compact.
\item $V_i (0,\e), V_1 (0,\e)$ is a dense couple of $T_\e M\zeq\G$.
\end{enumerate}
\end{enumerate}

Let $h\co M\zfl B$ be a function vanishing on $B\smallsetminus\zung_{i\zpe I}R_i$
such that for every $i\zpe I$:
\begin{enumerate} [label={\rm (\alph{*})}]
\item $j^4_i h =0$ but $j^5_i h\znoi 0$.
\item $h\co R_i \zfl\RR$ has compact support.
\end{enumerate}

We define $X_1 =V_1 +(h\zci\zp)X$,  which is obviously a complete and left-invariant vector field.

Then, for every $F\zpe\Aut(X,X_1 )$, we have $F\big(\zp^{-1}(R_i )\big)=\zp^{-1}(R_i )$ for all $i \in I$, and $F\co\zp^{-1}(R_i ) \zfl\zp^{-1}(R_i )$ preserves both $X$ and $X_1$.
Making use of a suitable identification of $\zp^{-1}(R_i )$ and
$\RR^k \zpor\GG$, our $F$ can be seen as a self-diffeomorphism of $\RR^k \zpor\GG$
which preserves $X=b_i \sum_{j=1}^k x_j \zpar/\zpar x_j +V_i$,
$b_i >0$, and $X_1 =V_1 +(h\zci\zp_1 )X$.

By Proposition \ref{P32} there exist $\zf_i \zpe \GL(k,\RR)$ and $\boldsymbol{\zl}_i \zpe\GG$ such that
$F(x,\g)=(\zf_i (x),\boldsymbol{\zl}_i \zpu\g)$. From Remark \ref{R44} if $k\zmai 2$,
or directly if $k=1$, it follows that $\zf_i$ is a multiple of the identity.
Finally, Remark \ref{R33} (note that $i$ is identified to $0\zpe\RR^k$) allows us to infer
that $\zf_i =\Id$, that is $F=L_{\boldsymbol{\zl}_i}$ on $\zp^{-1}(R_i )$.

Consider a $\GG$-invariant Riemannian metric on $M$. Then $F$ is an isometry on
$\zung_{i\zpe I}\zp^{-1}(R_i )$ and, by continuity, on $M$. Thus fixed $q\zpe I$ then
$F=L_{\zl_q}$ on the open set $R_q$, which is non-empty. Since $M$ is connected and
$F$, $L_{\boldsymbol{\zl}_q}$ are isometries, it follows that $F=L_{\boldsymbol{\zl}_q}$ on the whole $M$.

{\em In other words, Theorem \ref{T11} is proved in the free case ($k\zmai 1$).}

\begin{remark}\mylabel{R45}
Consider two $\GG$-invariant, positive and bounded function $\zf,\zf_1 \co M\zfl\RR$, and define
 $\widetilde{X}=\zf X$ and $\widetilde{X}_1 =\zf_1 X_1$ where $X,X_1$ are the vector
fields constructed above. Then both $\widetilde X$ and $\widetilde X_1$ are $\GG$-invariant and
complete, and the couple $\widetilde{X},\widetilde{X}_1$  describes $\GG$ as $X,X_1$ does so.
Indeed, $X$ and $\widetilde X$ share the same singularities with the same order and have identical trajectories, albeit with different parametrizations but maintaining the same orientation. The same holds for $X_1$ and  $\widetilde{X}_1$. Therefore, the reasoning involving $\widetilde{X}$ and $\widetilde{X}_1$ follows a similar line of argumentation as that for $X$ and $X_1.$
\end{remark}

\section{Actions with a free point} \mylabel{sec-5}
In this section we conclude the proof of Theorem \ref{T11}. Consider an action of a
connected compact Lie group $\GG$, with dimension $n$, on a connected $m$-manifold $M$.
Assuming the existence of a free point, the type of the principal orbits of the action
is $\GG$ itself, and if $S$ is the set of non-free points, then $M\smallsetminus S$ is open, connected and
dense. Furthermore, the $\GG$-action on $M\smallsetminus S$ is free (see \cite[Theorem IV.3.1]{BR}).

{\em We first assume} $m>n$. Let $X',X_1 '$ be a couple of vector fields on $M\smallsetminus S$ describing
$\GG$ and constructed like in the preceding section.

By \cite[Proposition 7.2]{TV2} there are two functions $\zr,\zf_1 \co M\zfl\RR$
bounded, $\GG$-invariant,  vanishing on $S$ and positive on $M\smallsetminus S$, such that the vector fields
$\widehat X,\widehat X_1$ on $M$ defined by $\widehat X=\zr X'$,
$\widehat X_1 =\zf_1 X'_1$ on $M\smallsetminus S$ and $\widehat X=\widehat X_1 =0$ on $S$, are
$\GG$-invariant and differentiable.

Set $X_1 =\widehat X_1$; note that $X_1 =\zf_1 X_1'$ on $M\smallsetminus S$.

Let $h\co\RR\zfl\RR$ be the function given by $h(t)=0$ if $t\zmei 0$ and
$h(t)=\exp(-1/t)$ if $t>0$. Set $X=(h\zci\zr)\widehat X$. Then $X=\zf X'$ on $M\smallsetminus S$ where
$\zf=(h\zci\zr)\zpu\zr$. Moreover, $X,X_1$ are $\GG$-invariant and complete, and $X$
vanishes at order infinity at $p$ if and only if $p\zpe S$. By Remark \ref{R45} the couple $X,X_1$ describes $\GG$ on $M\smallsetminus S$.

Given $F\zpe\Aut(X,X_1 )$, then $F(S)=S$ and $F\co M\smallsetminus S\zfl M\smallsetminus S$ is an automorphism of $X,X_1$ regarded as vector fields on $M\smallsetminus S$. Therefore,
there exists $\boldsymbol{\za}\zpe\GG$ such that $F=L_{\boldsymbol{\za}}$ on $M\smallsetminus S$ and, by continuity, on the whole $M$.

{\em Finally, we assume} $m=n$, that is we assume that $\GG$ acts on itself on the left.
Consider a couple dense $X,X_1$ of $\G$. Let $F\co\GG\zfl\GG$ be an automorphism of
this couple. By replacing $F$ by $L_{\boldsymbol{\zl}}\zci F$, where ${\boldsymbol{\zl}}=(F(\e))^{-1}$, one may
suppose $F(\e)=\e$. In this case $F_{\zbv\HH}=\Id$, where $\HH$ is the connected Lie
subgroup of $\GG$ whose Lie algebra is spanned by  $X,X_1$.
But $\overline\HH=\GG$, so $F=\Id$.

{\em This concludes the proof of Theorem \ref{T11}.}

We end this section by providing an example of an action with free points.

\begin{example}\mylabel{E62}
Endow $V=\zpr_{j=1}^{n-1}\RR^n $, $n\zmai 3$, with the $\operatorname{SO}(n)$-action given by
$\g\zpu v=(\g\zpu v_1 ,\dots,\g\zpu v_{n-1}),$ where $v=(v_1 ,\dots,v_{n-1})$. This
induces an action of $\operatorname{SO}(n)$ on the sphere
$$S^{n(n-1)-1}=\{v\zpe V\co \|v_1 \|^2 +\dots+\|v_{n-1} \|^2 =1 \},$$
such that a point $v$ of $S^{n(n-1)-1}$ is a free for this $\operatorname{SO}(n)$-action if and only if the vectors
$v_1 ,\dots,v_{n-1}$ are linearly independent. Thus the $\operatorname{SO}(n)$-action on $S^{n(n-1)-1}$ has both free and
non-free points.

For instance, let $n=3$, so $V=\RR^3 \zpor\RR^3$ and the induced $\operatorname{SO}(3)$-sphere is $S^5$. Let $F\co S^5 \zfl\RR^2$ be the map defined by
$F(v)=(\langle v_1 ,v_1\rangle, \langle v_1 ,v_2\rangle)$. Then a computation shows that $F(S^5 )=E$,
where $E=\{x\zpe\RR^2 \co (x_1 -1/2)^2 +x^2_2 \zmei1/4\}$, and $v\zpe S^5$ is free if
and only if $F(v)$ belongs to the interior of $E$. Indeed, given real numbers $a,b\geq 0$, and $c$, there exist two vectors
$v_1 , v_2 \zpe\RR^3$ such that $\| v_1 \|^2 =a$, $\| v_2 \|^2 =b$ and $<v_1 ,v_2 >=c$
if and only if $c^2 \zmei ab$, and $v_1 ,v_2$ are linearly independent if and only if $c^2 <ab$. Therefore, $x\zpe\RR^2$ belongs to
$F(S^5 )$ if and only if $x^{2}_2 \zmei x_1 (1-x_1 )$ while
$x$ is the image of a free point if and only if $x^{2}_2 < x_1 (1-x_1 )$. Finally observe that $x^{2}_2 \zmei x_1 (1-x_1 )$ if and only if
$ (x_1 -1/2)^2 +x^2_2 \zmei1/4$.
\end{example}

\section{Invariant vector fields and homogeneous spaces} \mylabel{sec-7}
In this section, we demonstrate that the hypothesis of the existence of a free point in Theorem \ref{T11} cannot be omitted. This is shown through Corollary \ref{C72} and Theorem \ref{T73}.

Now, consider an effective and transitive action
$\GG\zpor P\zfl P$ of a compact and connected Lie group $\GG$, of dimension $n$, on a
compact and connected manifold $P$, of dimension $r\zmai 1$.

Let $\A\zco\mathfrak X(P)$ be the Lie algebra of fundamental vector fields associated to
the action of $\GG$, thus $\A$ is  isomorphic,  in a natural way, to the algebra of right invariant
vector fields on $\GG$ (see \cite{PA}). Additionally, let $\B\zco\mathfrak X(P)$ be the Lie algebra
of those vector fields on $P$ that commute with (the action of) $\GG$, i.e., they are
$\GG$-invariant. It is evident that $[\A,\B]=0$.

For each $p\zpe P,$ let $\GG_p$ denote its isotropy subgroup, and set
$\B(p)=\{X(p) |\, X\zpe\B\}$. As $\GG$ and $\B$ commute one has:
\begin{itemize}
\item Every element of $\B$ is completely determined by its value at a point; thus
$\dim \B(p)=\dim \B$, $p\zpe P$, and $\B$ defines a foliation, say $\F$, on $P$.
\item If $p$ and $q$ belong to the same leaf of $\F$ then $\GG_p =\GG_q$.
\end{itemize}

Therefore, if $\dim \B=r$, then  $\GG_p$ does not depend on $p$, and since the action is
effective, $\GG_p =\{\e\}$. In other words, $P$ may be identified to $\GG,$ and $\B$ to
the algebra of left-invariant vector fields on $\GG$.
{\em Thus $\dim \B<r$ if and only if the isotropy subgroups have two or more elements.}

On the other hand, observe that if $\GG$ is a torus then necessarily each
$\GG_p =\{\e\}$ and $\dim \B=\dim \GG=\dim P$.

\begin{proposition}\mylabel{P71}
Let 
$\GG\zpor P\zfl P$ be an effective and transitive action of a compact and connected Lie group $\GG$, of dimension $n$, on a
compact and connected manifold $P$, of dimension $r\zmai 1$. Let $\B\zco\mathfrak X(P)$ be the Lie algebra vector fields on $P$ that commute with (the action of) $\GG$. If $\dim \B<r$, then there exists a diffeomorphism $\zl\co P\zfl P$ which preserves every
element of $\B$ and such that $\zl\znoi L_\g$ for all $\g\zpe\GG$.
\end{proposition}

\begin{proof}
We assume $\dim \B\zmai 1$ since the case $\B=0$ is trivial. First we show that the
leaves of $\F$ are compact regular submanifolds. Endow $P$ with a $\GG$-invariant
Riemannian metric $\tilde g$. Given any $p\zpe P$ there are normal coordinates on an
open neighborhood $A$ of $p$, $\ze>0$, and an integer $1\zmei\zlma\zmei r-1$ (see below)
such that $A$ is identified to $B_\ze (0)$, $p$  to the origin and $\GG_p$ to a subgroup
of the orthogonal group $\operatorname{O}(r)$ whose set of fixed points in $B_\ze (0)$ equals
$B_\ze (0)\zin(\{0\}\zpor\RR^\zlma)$ when one sets
$\RR^r =\RR^{r-\zlma}\zpor\RR^\zlma$.

Indeed, first observe that one may assume that $\tilde g (p)$, $p\zeq 0$, equals the standard scalar
product on $\RR^r$. Take any $\g\zpe\GG_p$. As $\g$ is an isometry of $(P,\tilde g)$, it maps
geodesics into geodesics, which implies that, in normal coordinates around $p$, the map
$\g$ is homogeneous of degree one, hence linear (see Remark \ref{Nuevo1}). Since
$\tilde g (p)$ is the scalar product of $\RR^r$, here linear means that ``$\g$ belongs to
$\operatorname{O}(r)$''.
Finally note that the set of fixed points of a subgroup of $\operatorname{O}(r)$ is always
a vector subspace of $\RR^r$.

Note that given $Y\zpe\B,$ then $Y(p)\zpe T_p P$ is invariant under the linear action of
$\GG_p$ on $T_p P$. Conversely, if $v\zpe T_p P$ is invariant under this linear action, then
one can define $Y\zpe\B$, with $Y(p)=v$, by setting $Y(\g\zpu p)=(L_\g)_* v$ (the linear
invariance of $v$ implies that this construction is correct). Therefore, in normal coordinates
$\B(p)=\{0\}\zpor\RR^\zlma\zco T_0 P$ and $\dim \B=\zlma$ (for this reason
$1\zmei\zlma\zmei r-1$ as announced above).

Let $N$ be the leaf of $\F$ passing through $p$. As $\GG_p =\GG_q$ for all $q\zpe N$,
then $N\zin A\zeq N\zin B_\ze (0)\zco(\{0\}\zpor\RR^\zlma )\zin B_\ze (0)$ and, if $\ze$
is small enough, $N\zin B_\ze (0)=(\{0\}\zpor\RR^\zlma )\zin B_\ze (0)$ since
$T_p N=\B(p)$. Therefore, $N$ is a regular submanifold (reasoning as before with every
point of $N$).

Consider a leaf $N'$ of $\F$ and a sequence $\{p_k \}_{k\zpe\NN}\zco N'$ with limit $p$;
then $N'=N$, that is $p\zpe N'$. Indeed, let $\GG'$ be the isotropy subgroup of any point in
$N'$ (all the points in $N'$ have the same isotropy  subgroup). From $\{p_k \}_{k\zpe\NN}\zfl p$
follows that $\GG'\zco\GG_p$. As $\GG'$ and $\GG_p$ are conjugated they have the same
dimension and the same number of connected component, so $\GG'=\GG_p$.

If $k$ is big enough and $\ze$ sufficiently small, since $\GG_p \zpu p_k =p_k$, one has that
$p_k \zpe(\{0\}\zpor\RR^\zlma )\zin B_\ze (0)= N\zin B_\ze (0)$; that is $p_k \zpe N$
and $N'=N$. This reasoning applied to the points of the closure of $N'$ shows that the leaves
of $\F$ are closed, so compact.

On the other hand, there are a simply connected Lie group $\HH$ and an action
$\HH\zpor P\zfl P$ whose algebra of fundamental vector fields equals $\B$ (see \cite{PA}).
Moreover, the actions of $\GG$ and $\HH$ commute because by construction $\GG$ and
$\B$ do so.

By the transitivity of the action of $\GG$, all the isotropy subgroups of the action of $\HH$ are
equal. Denote $\HH'$ this subgroup; then $\HH'$ is the kernel of the morphism
$\h\zpe\HH\mapsto L_\h \zpe\Diff(P)$. Therefore, the quotient group
$\widetilde\HH=\HH/\HH'$ acts freely on $P$ with $\B$ as algebra of fundamental
vector fields. Besides, the orbits of the action of $\widetilde\HH$ are the leaves of $\F$, so
$\widetilde\HH$ is compact and connected.

Thus the action of $\widetilde\HH$ gives rise to a $\widetilde\HH$-principal fibre bundle
$\zp\co P\zfl Q$. Observe that $\dim Q=r-\zlma$, so $1\zmei \dim Q <\dim P$
since $1\zmei\zlma\zmei r-1$.

Consider any $q\zpe Q$ and an open neighborhood $B$ of it such that
$\zp\co\zp^{-1}(B)\zfl B$ is identified, as $\widetilde\HH$-principal fibre bundle, to
$\zp_1 \co B_a (0)\zpor\widetilde\HH\zfl B_a (0)$ where $q\zeq 0$ and $a>0$. Let
$\zm\co B_a (0)\zfl\widetilde\HH$ be a map such that $\zm(0)\znoi\tilde\e$ and
$\zm(q')=\tilde\e$ outside of a compact $0\zpe K\zco B_a (0)$, where $\tilde\e$ is the
neutral element of $\widetilde\HH$. One define $\zl\co P\zfl P$ by setting $\zl=\Id$ on
$P\smallsetminus\zp^{-1}(B)$ and $\zl(x,\h)=(x,\h\zpu\zm(x))$ on
$\zp^{-1}(B)\zeq B_a (0)\zpor\widetilde\HH$.

With respect to the metric $\tilde g$, $\zl$ is not an isometry because equals the identity
on a non-empty open set but $\zl\znoi \Id$. Therefore, $\zl\znoi L_\g$ for all $\g\zpe\GG$
since each $L_\g$ is an isometry.

Moreover, $\zl$ preserves each $Y\zpe\B$ because $Y_{|\zp^{-1}(B)}$, where
$\zp^{-1}(B)\zeq B_a (0)\zpor\widetilde\HH$, is a fundamental vector field of the action
of $\widetilde\HH$ on $B_a (0)\zpor\widetilde\HH$ given by $\h'\zpu(x,\h)=(x,\h'\zpu\h)$
(as we pointed out before, the fundamental vector fields of the left action of a Lie group
on itself are the right invariant vector fields, see \cite{PA}).
\end{proof}

Finally the key result in this section is:

\begin{corollary}\mylabel{C72}
Let $Q$ be a connected manifold. Under the hypotheses of Proposition \ref{P71}, the action
of $\GG$ on $Q\zpor P$ given by $\g\zpu(q,p)=(q,\g\zpu p)$ is not determined by any
family of $\GG$-invariant vector fields.
\end{corollary}

\begin{proof}
Consider any $\GG$-invariant vector field $X$ on $Q\zpor P$. Then $X$ is foliated with respect
to the foliation given by the second factor, that is the foliation given by the orbits of $\GG$.
Therefore, $X=Y+V$ where $Y$ is s-horizontal, $V$ vertical and each $V(q,\_)$, $q\zpe Q$,
is $\GG$-invariant, that is belongs to $\B$.

Define the diffeomorphism $\tilde\zl\co Q\zpor P\zfl Q\zpor P$ by setting
$\tilde\zl(q,p)=(q,\zl(p))$ where $\zl$ is like in Proposition \ref{P71}. Clearly
$\tilde\zl_* X=X$.

Thus if $\mathcal L$ is a family of $\GG$-invariant vector fields, then
$\tilde\zl\zpe\Aut(\mathcal L)$ but $\tilde\zl\znope\GG\zco\Diff(Q\zpor P)$.
\end{proof}

From Corollary \ref{C72} it immediately follows:

\begin{theorem}\mylabel{T73}
Consider an homogeneous space $\GG/\HH$ and the natural left action of $\GG$ on it.
Assume that:
\begin{itemize}
\item $\GG$ is compact and connected.
\item $\HH$ has two or more elements and does not contain any normal subgroup of $\GG$
but $\{\e\}$.
\end{itemize}

Let $Q$ be a connected manifold. Then the action of $\GG$ on $Q\zpor(\GG/\HH)$ given by
$\g\zpu(q,[\g'])=(q,[\g\g'])$ is effective and cannot be described by any family of
$\GG$-invariant vector fields.
\end{theorem}

We finish this section with an example that illustrates an effective $\GG$-action on a manifold $M$ satisfying the following conditions: 
\begin{enumerate}[label={\rm (\arabic{*})}]
\item The set of fixed point $M^\GG$ is not trivial, indicating that the $\GG$-action is not transitive,

\item The $\GG$-action on $M$ is determined by a couple of vector field.

\item However, the induced $\GG$-action on $M\smallsetminus M^\GG$ cannot be determined by any family of $\GG$-invariant vector
fields.
\end{enumerate}

\begin{example}\mylabel{E74}
Consider the space $M = \mathbb{R}^{2n}$, where $n \geq 2$, equipped with coordinates $x = (x_1, \ldots, x_{2n})$. We introduce the complex structure $J=\zsu_{j=1}^n (e_{2j}\zte e^*_{2j-1}-e_{2j-1}\zte e^*_{2j})$, where ${e_1, \ldots, e{2n}}$ is the canonical basis of $\mathbb{R}^{2n}.$ With respect to $J,$
the space $M$ becomes a complex vector space of complex dimension $n,$ defined 
by $(a+bi)x=ax+bJx$. Therefore, $M$ is equipped with an effective action of the unitary group $\operatorname{U}(n),$ defined as the set of
elements in $\operatorname{GL}(n,\CC)$ that preserve the Hermitian inner product 
$\zh(x,y)=\langle x,y\rangle-i\langle Jx,y\rangle$. It is evident that $M^{\operatorname{U}(n)}=\{0\}.$

We show that the $\operatorname{U}(n)$-action on $M$ can be described by a couple of vector field.

Let $Y$ be the linear vector field associated to $J$, that is $Y(x)=Jx$, $x\zpe M$. Then
$Y=\zsu_{j=1}^n (-x_{2j}\zpar/\zpar x_{2j-1} +x_{2j-1}\zpar/\zpar x_{2j})$.

A real endomorphism $A$ of $M$ is $\CC$-linear if and only if $A\zci J=J\zci A$, which
is equivalent to say that $A$, as differentiable map, preserves the vector field $Y$.

Set $X=\zx=\zsu_{\zlma=1}^{2n}x_\zlma \zpar/\zpar x_\zlma$ and $X_1 =(\| x\|^2 -1)Y$.
Consider a diffeomorphism $f\co M\zfl M$ which preserves $X$; then, as it is
well known, $f\zpe \operatorname{GL}(2n,\RR)$. Now suppose that $f$ preserves $X_1$ too.
Then  $f$ maps $X_1^{-1}(0)$ into itself, so $f(S^{2n-1})=S^{2n-1}$.
Thus if the length of $x$ equals $1$, then that of $f(x)$ equals $1$ too, and
from the linearity of $f$ it follows that $\| f(x)\|=\| x\|$, $x\zpe M$.
In other words, $f$ is an isometry.
This last fact implies that the function $\|x\|^2 -1$
is $f$-invariant; therefore $f$ preserves $Y$, which means that $f$ is $\CC$-linear.
Thus $f\zpe \operatorname{O}(2n)\zin \operatorname{GL}(n,\CC)=\operatorname{U}(n)$.

Conversely, if $f\zpe\operatorname{U}(n),$ then $f$ preserves $X$ and $X_1$. Hence
$\Aut(X,X_1 )=\operatorname{U}(n)$.

Note that in our case, the origin of $0\in M$ plays a crucial role since the action of $\operatorname{U}(n)$
on $M\smallsetminus\{0\}=M\smallsetminus M^{\operatorname{U}(n)}$, $n\zmai 2$, cannot be determined by means of $\operatorname{U}(n)$-invariant vector
fields.

Indeed, consider the diffeomorphism $\zq\co M\smallsetminus\{0\} \zfl\RR^+ \zpor S^{2n-1}$
given by $\zq(x)=(\|x\|,x/\|x\|)$. Then the action of $\operatorname{U}(n)$
on $M\smallsetminus\{0\}$
becomes the action of $\operatorname{U}(n)$ on $\RR^+ \zpor S^{2n-1}$ defined by
$\g\zpu(t,y)=(t,\g\zpu y)$ and Corollary \ref{C72} applies.
\end{example}

\section{Some examples on compact connected linear groups} \mylabel{sec-8}
Both the natural action of $\operatorname{SO}(2)$ on $\RR^2$ and
the natural action of $\operatorname{SU}(2)$ on $\CC^2 \zeq  \RR^4$ have a free point. Therefore, these actions are determinable by Theorem \ref{T11}.
However we shall see that the natural actions of $\operatorname{SO}(n)$ on $\RR^n$, $n\zmai 3$, and of
$\operatorname{SU}(m)$ on $\CC^m \zeq  \RR^{2m}$, $m\zmai  3$,
are not determinable.

Indeed, let us start with $\operatorname{SO}(n)$, $n\zmai 3$. First observe
that at any $x\zpe\RR^n \smallsetminus\{0\}$ the only direction that is invariant
under the action of the isotropy group of $x$ is given by $\zx(x)$,
$\zx=\zsu_{j=1}^k x_j \zpar/\zpar x_j$. Therefore, if $X$ is an
$\operatorname{SO}(n)$-invariant vector field on
$x\zpe\RR^n \smallsetminus\{0\}$, then $X=f\zx$ where $f$ is an
$\operatorname{SO}(n)$-invariant function, that is to say constant on each
sphere centered at the origin.

This last fact implies that $f$ is $\operatorname{O}(n)$-invariant too and, as a
consequence, that $X$ is $\operatorname{O}(n)$-invariant too. Since the origin is a
fixed point, necessarily both the natural actions of $\operatorname{O}(n)$ and $\operatorname{SO}(n)$
have  the same set of invariant vector fields on $\RR^n$. Therefore, the natural action of
$\operatorname{SO}(n)$ on $\RR^n$, $n\zmai 3$, cannot be determined by means of
$\operatorname{SO}(n)$-invariant vector fields.

Now consider the case of $\operatorname{SU}(m)$, $m\zmai 3$. As in Example \ref{E74}, set
$Y=\zsu_{j=1}^m (-x_{2j}\zpar/\zpar x_{2j-1} +x_{2j-1}\zpar/\zpar x_{2j})$. Then the isotropy group of each $x\znoi 0$ has
just one invariant vector complex line that, as real plane, has  $\{\zx(x),Y(x)\}$
as a basis. Since $Y$ is invariant under the action of  $\operatorname{U}(m)$,
every  $\operatorname{SU}(m)$-invariant vector field $X$ on
$\RR^{2m} \smallsetminus\{0\}$ writes $X=f\zx +gY$ where $f$ and $g$  are
$\operatorname{SU}(m)$-invariant functions (and therefore
$\operatorname{U}(m)$-invariant). Finally the same reasoning as
before shows that $\operatorname{SU}(m)$ cannot be described
by means of $\operatorname{SU}(m)$-invariant vector fields, although the natural action of $\operatorname{U}(m)$ can be
determined by two invariant vector fields (Example \ref{E74}).

Our next example shows that the natural action of the symplectic group can be described by a couple of invariant
vector field.

\begin{example}\mylabel{E81}
On $\RR^{4r}$, $r\zmai 1$, let consider the following complex structures
$$J=\zsu_{\zlma=1}^r \left(e_{4\zlma-2}\zte e^*_{4\zlma-3}-e_{4\zlma-3}\zte e^*_{4\zlma-2}
+e_{4\zlma}\zte e^*_{4\zlma-1}-e_{4\zlma-1}\zte e^*_{4\zlma}\right)$$
$$K=\zsu_{\zlma=1}^r \left(e_{4\zlma-1}\zte e^*_{4\zlma-3}-e_{4\zlma}\zte e^*_{4\zlma-2}
-e_{4\zlma-3}\zte e^*_{4\zlma-1}+e_{4\zlma-2}\zte e^*_{4\zlma}\right)$$
$$L=\zsu_{\zlma=1}^r \left(e_{4\zlma}\zte e^*_{4\zlma-3}+e_{4\zlma-1}\zte e^*_{4\zlma-2}
-e_{4\zlma-2}\zte e^*_{4\zlma-1}-e_{4\zlma-3}\zte e^*_{4\zlma}\right)$$
where $\{e_1 ,\dots,e_{4r}\}$ is the canonical base of $\RR^{4r}$.

With respect to $J$, $K$ and $L$
the space $\RR^{4r}$ becomes a quaternionic (left) vector space of dimension $r$
by setting $(a+bi+cj+dk)x=ax+bJx+cKx+dLx$.

Let  $Y,Z,U$ be the linear vector fields on $\RR^{4r}$ given by $Y(x)=Jx$,
$Z(x)=Kx$ and $U(x)=Lx$, $x\zpe\RR^{4r}$, respectively. A real endomorpism $A$
of $\RR^{4r}$ is $\mathbb H$-linear if and only if $A\zci J=J\zci A$, $A\zci K=K\zci A$
and $A\zci L=L\zci A$, which is equivalent to say that $A$, as differentiable map,
preserves $Y$, $Z$ and $U$.
Thus the symplectic group $\operatorname{Sp}(r)$, $r\zmai 1$, is the set of those
$A\zpe\operatorname{O}(4r)$ that preserves $Y$, $Z$ and $U$.

Let consider the open intervals $I_1 =(4,9)$, $I_2 =(16,25)$ and $I_3 =(36,49)$, and let 
$\zf_1 ,\zf_2 ,\zf_3 \colon\RR\zfl\RR$ be functions such that:
\begin{enumerate}[label={\rm (\arabic{*})}]
\item\mylabel{NU1} $\zf_1^2 +\zf_2^2 +\zf_3^2 >0$ on $\RR \smallsetminus\{1\}$
and $\zf_1 (1)=\zf_2 (1)=\zf_3 (1)=0$.
\item\mylabel{NU2} $\zf_a (I_b )=\zd_{ab}$ for every $a=1,2,3$ and every $b=1,2,3$.
\end{enumerate}

Now set $X=\zx$ and $X_1 (x)=\zf_1 (\| x\|^2 )Y(x)+\zf_2 (\| x\|^2 )Z(x)
+\zf_3 (\| x\|^2 )U(x)$, $x\zpe\RR^{4r}$, which defines a vector field $X_1$ on
$\RR^{4r}$. Observe that $X_1^{-1}(0)=\{0\} \zun S^{4r-1}$.

Then $\operatorname{Sp}(r)$ equals $\operatorname{Aut}(X,X_1 )$.
Indeed, obviously $\operatorname{Sp}(r)\zco\operatorname{Aut}(X,X_1 )$.
Conversely, consider any  $f\zpe\operatorname{Aut}(X,X_1 )$. As $f$ preserves $X$,
necessarily it is $\RR$-linear. On the other hand as $f$ preserves $X_1$ then
$f(X_1^{-1}(0))=X_1^{-1}(0)$, so $f(S^{4r-1})=S^{4r-1}$ and $f$ is an isometry
(see Example \ref{E74}).

Observe that $X_1 =Y$ on  $B_3 (0) \smallsetminus \overline{B}_2 (0)$. Therefore,
$f$ preserves $Y$ on this open set and, by the analyticity of $f$ and $Y$,
on the whole $\RR^{4r}$. Analogously $f$ preserves $Z$ and $U$,
hence $f\zpe\operatorname{Sp}(r)$.
\end{example}

We now consider a generic manifold $M$ of dimension $m$, and let $\operatorname{Act}(M)$ be
the set of all the effective actions of compact connected Lie groups on $M$.
Thinking of these actions as subgroups of $\operatorname{Diff}(M)$ gives rise,
by inclusion, to a partial order on  $\operatorname{Act}(M)$. Finally, let
 $\operatorname{Act}_0(M)$ be the set of those elements of  $\operatorname{Act}(M)$
 that are determined by some family of invariant vector fields.
 
The poset $\operatorname{Act}_0(M)$ has nice structural properties:

\begin{lemma}\mylabel{lem:acc}
The ascending chain condition holds in $\operatorname{Act}_0(M)$. In particular,  every element in  $\operatorname{Act}_0(M)$  is included in some maximal
element.
\end{lemma}
\begin{proof}
Recall that elements in  $\operatorname{Act}_0(M)$ can be regarded as connected compact groups of isometries for a suitable Riemannian metric on $M$. Consequently, the 
dimension of any $\GG$
in $\operatorname{Act}_0(M)$ is bounded above by $m(m+1)/2$ by Theorems 3.3 and 3.4 in \cite[Chapter VI]{KN}. 

Now, consider any strictly ascending sequence
$$\GG_1\subset\GG_2\subset\ldots\subset \GG_r\subset$$
in $\operatorname{Act}_0(M).$ Since every group is connected, we have
$0<\operatorname{dim}\GG_i < \operatorname{dim}\GG_{i+1}\leq m(m+1)/2$, and this sequence must eventually stop. Moreover, if the sequence cannot be extended, it must conclude with a maximal element.
\end{proof}


A direct consequence of the ascending chain condition is the following stability property for the maximal elements in $\operatorname{Act}_0(M)$:

\begin{proposition}\mylabel{P82}
Let $\GG$ be a maximal element in $\operatorname{Act}_0 (M)$, and let $\mathcal W$ be
a family of $\GG$-invariant vector fields on $M$. Then either $\mathcal W$ does not
determine any element of $\operatorname{Act}_0 (M)$ or
 $\GG=\operatorname{Aut}(\mathcal W)$.
\end{proposition}


 \begin{remark}\mylabel{R83}
A particular case of the above proposition is as follows. Suppose that $\GG$ is a maximal element in $\operatorname{Act}_0 (M)$ and is
determined by two $\GG$-invariant vector fields $X,X_1$. If we perturb these vector fields to obtain a new pair, $X',X'_1$, then one of two scenarios arises: either $X',X'_1$ do not determine any element in $\operatorname{Act}_0(M)$, or we have $\GG = \operatorname{Aut}(X', X'_1)$.
\end{remark}

 \begin{proposition}\mylabel{P84}
The natural action of $\operatorname{U}(m)$ on $\RR^{2m}$, $m\zmai 1$,
is maximal in $\operatorname{Act}_0(\RR^{2m})$.
\end{proposition}
\begin{proof}
Let us start recalling a result needed later on.
Consider two connected compact Lie subgroup $\HH,\HH'$ of $\operatorname{SO}(2m)$,
$\zmai 1$, the first one isomorphic to $\operatorname{U}(m)$. If
$\HH\zco\HH'$ then either $\HH=\HH'$ or
 $\HH= \operatorname{SO}(2m)$ \cite{AFG}.

Let $\GG\in\operatorname{Act}_0(\RR^{2m})$ such that
$\operatorname{U}(m)\zco\GG$. Then every $\GG$-orbit has
dimension strictly smaller than $2m$, as otherwise $\RR^{2m}$ would be compact. Furthermore, since $\operatorname{U}(m)\zco\GG$,  every $\GG$-orbit contains the induced $\operatorname{U}(m)$-orbit. This induced orbit is either a $(2m-1)$-sphere centered at the origin or the origin, which we identify as the degenerate sphere with radius zero. Combining both arguments, we conclude that the $\GG$-orbits consist of all the spheres centered at the origin, including the degenerate one, and the origin is a fix point of the $\GG$-action.

Endow $\RR^{2m}$ with a $\GG$-invariant Riemannian
metric $\tilde g$. For this metric consider normal coordinates $(u_1 ,\dots,u_{2m})$
in an open neighborhood $A$ of the origin, which is identified to some open ball
$B_\ze (0)$ (origin to origin). 
As usual we suppose that $\widetilde{g}(p),$ $p\equiv 0,$ equals the scalar product of $\RR^{2m}$.
Then both $\operatorname{U}(m)$ and $\GG$ are identified
to subgroups of $\operatorname{SO}(2m)$ (see the second paragraph of the proof of Proposition \ref{P71}).

If $m=1$ then $\operatorname{dim}\operatorname{U}(1)
=\operatorname{dim}\operatorname{SO}(2)=1$ and necessarily
$\GG=\operatorname{U}(1)$. Therefore, assume $m\zmai 2$ and
$\GG\znoi\operatorname{U}(m)$, which implies that $\GG$ is isomorphic to
$\operatorname{SO}(2m)$ and $\dim\GG=m(2m-1)$.

Moreover the action of $\GG$ on every
sphere of positive radius is effective. Indeed, if not the Lie algebra of $\GG$,
and so that of $\operatorname{SO}(2m)$, includes a proper ideal, which excludes
the case $m\zmai 3$. If $m=2$, as the action of $\operatorname{U}(2)$ on these
spheres is effective, this ideal ought to have codimension $\zmai 4$, but there is
no such ideal.

 Set $Y=\zsu_{j=1}^m (-x_{2j}\zpar/\zpar x_{2j-1} +x_{2j-1}\zpar/\zpar x_{2j})$
(see Example \ref{E74} again). At each $x\znoi 0$, and for the action of
$\operatorname{U}(m)$, the isotropy group of this point has just a vector real plane
of invariant vectors one of whose basis is $\{ \zx (x),Y(x)\}$. Therefore, each
$\operatorname{U}(m)$-invariant vector field $X$ on
$\RR^{2m}\smallsetminus \{0\}$ writes $X=f\zx+gY$ where $f$  and $g$ are
functions of $\| x\|^2$.

Let $S_x$, $x\znoi 0$, be the sphere centered at the origin and passing through $x$.
Since the action of $\GG$ is effective on all its orbits but $\{0\}$ and
$\dim\GG=\dim\operatorname{SO}(2m)$, a vector $v\zpe T_x \RR^{2m}$, $x\znoi 0$,
is invariant under the action of the isotropy group of $x$ if an only if it is
$\tilde g$-orthogonal to $T_x S_x$. Therefore, there exists a $\GG$-invariant
vector field $V$ on $\RR^{2m}\smallsetminus\{0\}$ such that:
\begin{enumerate}[label={\rm (\arabic{*})}]
\item\mylabel{NU3} $\tilde g(V,V)=1$.
\item\mylabel{NU4} $V$ is $\tilde g$-orthogonal to every $S_x$, $x\znoi  0$, and
points outward.
\end{enumerate}

Obviously $V$ is $\operatorname{U}(m)$-invariant; therefore, it writes
$V=a\zx+bY$ where $a$ and $b$ are functions of $\| x\|^2$ and $a>0$ everywhere.

Observe that any $\GG$-invariant vector field defined on a punctured open ball
$B_\zr (0)\smallsetminus\{0\}$ is the product of a function of $\| x\|^2$ and $V$.
Let $R$ be the radial vector field of the normal coordinates $(u_1 ,\dots,u_{2m})$, that
is to say $R=\zsu_{j=1}^{2m} u_j \zpar /\zpar u_j$. On a punctured open ball
$B_\zt (0)\smallsetminus\{0\}$ included in $A$ one has $R=fV$ where $f$ is a positive
function of $\| x\|^2$.

Consider a positive function $\zf$ of $\| x\|^2$ defined on
$\RR^{2m} \smallsetminus\{ 0\}$ such that $\zf V$ is complete and
$\zf=f$ on $B_{\zt /2} (0)\smallsetminus\{0\}$. Define a new vector field $W$ on
$\RR^{2m}$ by setting $W=R$ on $B_{\zt /2} (0)$ and $W=\zf V$
on $\RR^{2m}\smallsetminus\{ 0\}$. Then $W$ is complete at its linear part at the origin
equals the identity. Moreover, $W$ is $\GG$-invariant.

Now by Sternberg linearization theorem (see Theorem \ref{Nuevo2} and page 319
of \cite{TV1}) there exists a diffeomorphism $F\co\RR^{2m}\zfl\RR^{2m}$
such that $F_* W=\zx$.

Consider the conjugate action on $\RR^{2m}$, which will be called $\GG^*$, given by
$\g\zpu x=F(\g\zpu F^{-1}(x))$. Clearly $\GG^*$ belongs to
$\operatorname{Act}_0 (\RR^{2m})$ and $\zx$ is $\GG^*$-invariant; therefore,
$\GG^*$ is linear. As this group is compact, there always exists a $\GG^*$-invariant
scalar product. Thus via a linear automorphism of $\RR^{2m}$, the action of
$\GG^*$ can be assimilated to the natural action of $\operatorname{SO}(2m)$
on $\RR^{2m}$. But this last one does not belong to
$\operatorname{Act}_0 (\RR^{2m})$, {\em contradiction}.
\end{proof}

\section{Some open questions}\mylabel{Open}

In this last section we collect a list of open questions that naturally arise from our work.

First, we observe that our main result, Theorem \ref{T11}, can be expressed in terms of centralizers of elements within the full group of diffeomorphisms.

Indeed, let us $Z$ be a complete vector field $Z$ on a manifold $P$, and let $\zF_t$ be its corresponding flow. Consider two rationally independent real numbers, denoted as $a$ and $b$. Now, given a diffeomorphism $f\co P \zfl P$, it commutes with every $\zF_t$, where $t\zpe\RR$, if and only if it commutes with both $\zF_a$ and $\zF_b$.

Additionally, we observe that $f$ commutes with the flow $\zF_t$, that is $f\in C_{\Diff(M)}(\zF_t)$, if and only if it preserves the vector field $Z$. Therefore, Theorem \ref{T11} leads us to the following conclusion:

\begin{corollary}\mylabel{Nuevo3}
Consider an action of a connected compact Lie group $\GG$ on a connected manifold
$M$. If the action of $\GG$ possess a free point, then there exist $f_i\in \Diff(M)$, $i=1,\ldots,4$, such that each $f_i$ is diffeotopic to the identity map, and
$$\GG=\Interseccion\limits_{i=1}^4 C_{\Diff(M)}(f_i).$$
\end{corollary}

Then, it is natural to ask whether group actions of connected Lie groups on connected manifolds can be described as centralizer of diffeomorphisms:

\begin{question}\mylabel{QU1}
Consider an action of a connected compact Lie group $\GG$ on a connected manifold
$M$.
\begin{enumerate}[label={\rm (\alph{*})}]
\item\mylabel{QU1.1} Under which hypothesis is it possible to find a family of diffeomorphisms $\{f_i\in\Diff(M): i\in I\}$ such that $$\GG=\Interseccion\limits_{i\in I} C_{\Diff(M)}(f_i)?$$

\item\mylabel{QU1.2} If such a family exists, what is the minimum number of diffeomorphisms it must contain?

\end{enumerate}
\end{question}

In Section \ref{sec-8} it is shown that there exist connected linear compact groups whose natural action is not determinable by any family of invariant vector fields. Lemma \ref{lem:acc} and Proposition \ref{P82} suggest that an inductive argument via maximal subgroups of connected compact Lie groups, as classified in \cite{AFG}, may allow to tackle the following open question:

\begin{question}\mylabel{QU2}
For every integer $n > 0$, what is the complete description of all connected compact subgroups $\GG \subset \operatorname{GL}(n,\RR)$ such that the natural $\GG$-action on $\RR^n$ is determinable?
\end{question}

In a completely general setting, solving the problem equivalent to Question \ref{QU2}, which seeks the complete description of elements in $\operatorname{Act}_0(M)$, may seem unattainable for a generic $M$. Nevertheless, Proposition \ref{P82} and Remark \ref{R83} indicate that maximal elements in $\operatorname{Act}_0(M)$ could be characterized using invariant vector fields and a suitable definition of action stability. This leads to the following inquiry:

\begin{question}\mylabel{QU3}
Consider a connected $m$-manifold $M$ and a family of vector fields $\mathcal{F}$ on $M$. Is it possible to devise a suitable definition of stability for the $\text{Aut}(\mathcal{F})$-action in terms of the elements in $\mathcal{F}$ and characterize those actions that are stable?
\end{question}

Finally, our methods are not applicable in the $C^0$ class because they rely on Lemma 3.4 in \cite{TV1}, which does not hold for continuous maps. Similarly, they cannot be applied in the analytic case due to the use of plateau functions. Therefore, a natural question arises:

\begin{question}\mylabel{QU4}
Do our results hold true in the continuous or analytic classes?
\end{question}



\begin{thebibliography}{99}
\parskip=0.7pt


\bibitem{AFG}
{Antoneli, F., Forger, M. and Gaviria, P.},
{\em  Maximal subgroup of compact Lie groups},
{J.\ Lie\ Theory} \textbf{22}
(2012), 949--1024.


\bibitem{BR}
{Bredon, G.E.},
{``Introduction to compact transformation groups''},
Academic Press, Pure and Applied Mathematics, Vol.\ 46, 1972.


\bibitem{HI}
{Hirsch, M.W.}, {``Differential topology''},
Graduate Texts in Mathematics, No. 33. Springer-Verlag, New York-Heidelberg, 1976.

\bibitem{KN}
{Kobayashi, S. and Nomizu, K.},
{``Foundations of differential geometry Vol.\ I''},
Publishers, a division of John Wiley \& Sons, New York-London 1963.

\bibitem{PA}
{Palais, R. S.}, {\em A global formulation of the Lie theory of transformation groups},
{Mem. Amer. Math. Soc.}  \textbf{22} (1957) 1--123.



\bibitem{SU}
{Schreier, J. and Ulam, S.},
{\em Sur le nombre de g\'en\'erateurs d'un groupe topologique compact et connexe},
{Fund. Math.} \textbf{24.1} (1935), 302--304.

\bibitem{SST}
{Sternberg S.},
{\em On the structure of local homeomorphisms of euclidean $n$-spaces II},
{Amer.\ J.\ Math.} \textbf{80} (1958), 623--631.

\bibitem{TV1}
{Turiel, F.-J. and Viruel, A.},
{\em Finite $C^{\zinf}$-actions are
described by a single vector field},
{Rev.\ Mat.\ Iberoam.} \textbf{30}
(2014), 317--330.

\bibitem{TV2}
{Turiel, F.-J. and Viruel, A.},
{\em Smooth torus actions are described by a single vector field},
{Rev.\ Mat.\ Iberoam.} \textbf{34}
(2018), 839--852.

\bibitem{WA}
{Warner, F. W.}, {``Foundations of  Differentiable Manifolds and Lie Groups''},
Graduate Texts in Mathematics, No. 94. Springer-Verlag, New York-Heidelberg, 1983.






\end{thebibliography}
\end{document}